\documentclass[11pt]{article}
\usepackage{amsmath,amssymb,latexsym}
\def\Z{{\mathbb Z}}
\def\bZ{{\mathbb Z}}

\def\Gal{{\rm Gal}}
\def\GL{{\rm GL}}
\def\SL{{\rm SL}}
\def\Cl{{\rm Cl}}

\def\Aut{{\rm Aut}}
\def\Stab{{\rm Stab}}
\def\Id{{\rm Id}}
\def\End{{\rm End}}
\def\eq{{\rm eq}}

\def\im{{\rm im}}

\def\Disc{{\rm Disc}}
\def\disc{{\rm disc}}

\def\proj{{\rm proj}}
\def\red{{\rm red}}

\def\R{{\mathbb R}}

\def\bR{{\mathbb R}}
\def\bC{{\mathbb C}}
\def\bF{{\mathbb F}}

\def\I{{\mathcal I}}
\def\Q{{\mathbb Q}}

\def\Z{{\mathbb Z}}

\def\Q{{\mathbb Q}}
\def\O{{\mathcal O}}
\def\cI{{\mathcal I}}
\def\bQ{{\mathbb Q}}

\def\cV{{\mathcal V}}
\def\bZ{{\mathbb Z}}

\def\bF{{\mathbb F}}
\def\bQ{{\mathbb Q}}
\def\cO{{\mathcal O}}
\def\cI{{\mathcal I}}

\def\cS{{\mathcal S}}
\def\cW{{\mathcal W}}

\newtheorem{theorem}{Theorem}
\newtheorem{lemma}[theorem]{Lemma}
\newtheorem{corollary}[theorem]{Corollary}
\newtheorem{proposition}[theorem]{Proposition}

\newtheorem{example}[theorem]{Example}
\newtheorem{remark}[theorem]{Remark}
\newenvironment{proof}{\noindent {\bf Proof:}}{$\Box$ \vspace{2 ex}}

\begin{document}

\title{The mean number of $3$-torsion elements in the class groups and
  ideal groups of quadratic orders}

\author{Manjul Bhargava and Ila Varma}

\maketitle

\begin{abstract}
  We determine the mean number of 3-torsion elements in the class
  groups of quadratic orders, where
  the quadratic orders are ordered by their absolute discriminants.
  Moreover, for a quadratic order $\O$ 
we distinguish between the two groups: $\Cl_3(\O)$, the group of
{\it ideal classes of order~$3$}; and $\I_3(\O)$, the group of {\it
  ideals of order~$3$}.  We determine the mean values of
both $|\Cl_3(\O)|$ and $|\I_3(\O)|$, as $\O$ ranges over any
family of orders defined by finitely many (or in suitable cases, even infinitely many) local conditions.
  
As a consequence, we prove the surprising fact 
that the mean value of the difference
$|\Cl_3(\O)|-|\I_3(\O)|$ is equal to 1, regardless of
whether one averages over the maximal orders in complex quadratic fields  
or over all orders in such fields or, indeed, over any family of
complex quadratic orders defined by local conditions.  For any family of
real quadratic orders defined by local conditions, we prove similarly 
that the mean value of the difference 
$|\Cl_3(\O)|-\frac13|\I_3(\O)|$ is equal to 1, independent of the family. 
\end{abstract}

\section{Introduction}

In their classical paper~\cite{DH}, Davenport and Heilbronn proved the
following theorem.
\begin{theorem}\label{theoremdh}
When quadratic fields are ordered by their absolute
discriminants:
\begin{itemize}
\item[{\rm (a)}]
The average number of $3$-torsion elements in the class
groups of imaginary quadratic fields~is~$2$.
\item[{\rm (b)}]
The average number of $3$-torsion elements in the class
groups of real quadratic fields is~$\textstyle{\frac{4}{3}}$.
\end{itemize}
\end{theorem}
This theorem yields the only two proven cases of the 
Cohen-Lenstra heuristics for class groups of quadratic fields. 

In their paper~\cite[p.\ 59]{CL}, Cohen and Lenstra raise the
question as to what happens when one looks at class groups over {\it all
orders}, rather than just the maximal orders corresponding to fields.  
The heuristics formulated by Cohen and Lenstra for class groups of quadratic 
fields are based primarily on the assumption that,
in the absence of any known structure for these abelian groups beyond
genus theory, we may as well
assume that they are ``random'' groups in the appropriate sense.  

For orders, however, as pointed out by Cohen and Lenstra
themselves~\cite{CL}, when an imaginary quadratic 
order is not maximal there is an
additional arithmetic constraint on the class group coming from the
class number formula.  Indeed, if $h(d)$ denotes the class number of the imaginary quadratic order of discriminant $d$, and if $D$ is a (negative) fundamental discriminant,
then the class number formula gives
\begin{equation}\label{cnf}
h(Df^2) = 
\Bigl[f \cdot\prod_{{p|f}}\left(1-\frac{(D|p)}{p}\right)\Bigr]
h(D),
\end{equation}
where $(\cdot | \cdot)$ denotes the Kronecker symbol.
Thus, one would naturally expect that the
{percentage} of quadratic orders having class number divisible by 3
should be strictly larger than the corresponding percentage for
quadratic fields. Similarly, the {average number} of 3-torsion
elements across all quadratic orders would also be expected to be
strictly higher than the corresponding average for quadratic
fields.\footnote{Note that the class number formula does not give
complete information on the number of 3-torsion elements; indeed, extra 
factors of 3 in the class number may mean 
extra 3-torsion, but it could also mean 
extra 9-torsion or 27-torsion, etc.!}

In this article, we begin by proving the latter statement:

\begin{theorem}\label{thmorders}
When orders in quadratic fields are ordered by their absolute
discriminants:
\begin{itemize}
\item[{\rm (a)}]
The average number of $3$-torsion elements in the class
groups of imaginary quadratic orders~is $\displaystyle{1+\frac{\zeta(2)}{\zeta(3)}}$.
\item[{\rm (b)}]
The average number of $3$-torsion elements in the class
groups of real quadratic fields is $\displaystyle{1+\frac{1}{3}\cdot\frac{\zeta(2)}{\zeta(3)}}$.
\end{itemize}
\end{theorem}
Note that $\frac{\zeta(2)}{\zeta(3)}\approx  1.36843>1$.

More generally, we may consider the analogue of
Theorem~\ref{thmorders} when the average is taken not over all orders,
but over some subset of orders defined by local conditions.  More
precisely, for each prime $p$, let $\Sigma_p$ be any set of
isomorphism classes of orders in \'etale quadratic algebras over
$\Q_p.$ We say that the collection $(\Sigma_p)$ of local
specifications is {\it acceptable} if, for all sufficiently large $p$, the
set $\Sigma_p$ contains all the maximal quadratic rings over $\Z_p.$
Let $\Sigma$ denote the set of quadratic orders $\O$, up to
isomorphism, such that $\O\otimes\Z_p\in\Sigma_p$ for all $p$.  Then
we may ask what the mean number of 3-torsion elements in the class
groups of imaginary (resp.\ real) quadratic orders in $\Sigma$ is.

To state such a result for general acceptable $\Sigma$, we need a bit of
notation.  For an \'etale cubic algebra $K$ over $\Q_p$, we write
$D(K)$ for the unique quadratic algebra over $\Z_p$ satisfying
$\Disc(D(K))=\Disc(K)$.  Also, for an order $R$ in an \'etale
quadratic algebra over $\Q_p$, let $C(R)$ denote the weighted number
of \'etale cubic algebras $K$ over $\Q_p$ such that $R\subset D(K)$:
	\begin{equation}\label{Cdef}
	C(R) := \sum_{{\mbox{\scriptsize $K$ \'etale cubic$/ \bQ_p$}}\atop{\mbox{\scriptsize s.t.  $R \subset D(K)$}}} \frac1{\#\Aut(K)}.
	\end{equation}  
We define the ``cubic mass'' $M_\Sigma$ of the
family $\Sigma$ as a product of local masses: 
\begin{equation}\label{massdef}
M_{\Sigma} ~:= \quad 
\prod_p\,\frac{{\displaystyle\sum_{R\in\Sigma_p}\frac{C(R)}{\Disc_p(R)}}}{{\displaystyle \sum_{R\in\Sigma_p}\frac{1}{\#\Aut(R)}\cdot\frac1{\Disc_p(R)}}} = \prod_p\, \frac{{\displaystyle \sum_{R\in \Sigma_p}\frac{C(R)}{\Disc_p(R)}}}{\displaystyle{\sum_{R \in \Sigma_p} {\frac{1}{2\cdot \Disc_p(R)}}}}\,,
\end{equation}
where $\Disc_p(R)$ denotes the discriminant of $R$ viewed as a power
of $p$. We then prove the following generalization of~Theorem~\ref{thmorders}.

\begin{theorem}\label{gensigmaord}
Let $(\Sigma_p)$ be any acceptable collection of local specifications
as above, and let $\Sigma$ denote the set
of all isomorphism classes of quadratic orders $\O$  such that 
$\O\otimes\Z_p\in\Sigma_p$ for all $p$.  
  Then, when orders in $\Sigma$ are ordered by their absolute
  discriminants:
\begin{itemize}
\item[{\rm (a)}]
The average number of $3$-torsion elements in the class
groups of imaginary quadratic orders in $\Sigma$ is 
$\displaystyle{1+M_\Sigma}$.
\item[{\rm (b)}]
The average number of $3$-torsion elements in the class
groups of real quadratic orders \linebreak in~$\Sigma$~is~$\displaystyle{1+\frac{1}{3}M_\Sigma}$.
\end{itemize}
\end{theorem}

If $\Sigma$ is
the set of all orders in Theorem~\ref{gensigmaord}, we recover
Theorem~\ref{thmorders}; if $\Sigma$ is the set of all maximal orders,
we recover Theorem~1. As would be expected, the mean number of 3-torsion 
elements in class groups of quadratic orders depends on which set of 
orders one is taking the average over.   However, a remarkable consequence 
of Theorem~\ref{gensigmaord} is the following generalization of Theorem~1:

\begin{corollary}\label{maxcase}
Suppose one restricts to just those quadratic fields satisfying any
specified set of local conditions at any finite set of primes.  Then,
when these quadratic fields are ordered by their absolute
discriminants:
\begin{itemize}
\item[{\rm (a)}]
The average number of $3$-torsion elements in the class
groups of such imaginary quadratic~fields is~$2$.
\item[{\rm (b)}]
The average number of $3$-torsion elements in the class
groups of such real quadratic fields is~$\textstyle{\frac{4}{3}}$.
\end{itemize}
\end{corollary}
Thus the mean number of 3-torsion elements in class groups of quadratic 
{fields} (i.e., of maximal quadratic orders) remains the
same even when one averages over families of quadratic fields
defined by any desired finite set of local conditions.
  
We turn next to 3-torsion elements in the {\it ideal group} of a 
quadratic order $\O$, i.e., the group $\I(\O)$ of invertible fractional ideals of $\O$, of which the class group
$\Cl(\O)$ is a quotient.  It may come as a surprise that, if a quadratic order
is not maximal, then it is
possible for an {\it ideal} to have order~$3$, i.e., $I$ can be 
a fractional ideal of a
quadratic order $\O$ satisfying $I^3=\O$, but $I\neq \O$.  We first illustrate this phenomenon with 
an example:

\begin{example}\label{ex1}
{\em Let $\O=\Z[\sqrt{-11}]$ and let $I=(2,\frac{1-\sqrt{-11}}{2})$.  Then
$I\subset\O\otimes\Q$ is a fractional ideal of $\O$ and has norm
one.  Since $I^3\subset\O$, and $I$ has norm one, we must
have $I^3=\O$, even though clearly $I\neq\O$.  
Hence $I$ has order 3 in the ideal
group of $\O$. 
It follows, in particular, that 
the {\it ideal class} represented by $I$ also has order 3 in the class group of $\O$!
}\end{example}

Example~\ref{ex1} shows that some elements of the ideal class group can have order 3 simply
because there exists a (non-principal) ideal representing them that has 
order 3 in the ideal group.  This raises 
the question as to how many 3-torsion elements exist in the ideal group on average in
quadratic orders.  For maximal orders, it is easy to show that any
3-torsion element (indeed, any torsion element) in the ideal group
must be simply the trivial ideal.  For all orders, we have the following theorem.

\begin{theorem}\label{sigmaid}
When orders in quadratic fields are ordered by their absolute discriminants,
the average number of $3$-torsion elements in the ideal groups of either imaginary or real quadratic orders is $\displaystyle\frac{\zeta(2)}{\zeta(3)}$.
\end{theorem}
 
In the case of general sets 
of orders defined by any acceptable set of local conditions, we have
the following generalization of Theorem~\ref{sigmaid}: 

\begin{theorem}\label{gensigmaid}
Let $(\Sigma_p)$ be any acceptable collection of local specifications
as above, and let $\Sigma$ denote the set
of all isomorphism classes of quadratic orders $\O$  such that 
$\O\otimes\Z_p\in\Sigma_p$ for all $p$.  
  Then, when orders in $\Sigma$ are ordered by their absolute
  discriminants:
\begin{itemize}
\item[{\rm (a)}]
The average number of $3$-torsion elements in the ideal
groups of imaginary quadratic orders in $\Sigma$ is~$\displaystyle{M_\Sigma}$.
\item[{\rm (b)}]
The average number of $3$-torsion elements in the ideal
groups of real quadratic orders \linebreak in~$\Sigma$~is~$\displaystyle{M_\Sigma}$.
\end{itemize}
\end{theorem}

In the preceding theorems, we have distinguished between the two
groups $\Cl_3(\O)$, the group of {\it ideal classes of order~$3$},
and $\I_3(\O)$, the group of {\it ideals of order~$3$}.
Theorems~\ref{gensigmaord} and \ref{gensigmaid} give the mean values
of $|\Cl_3(\O)|$ and $|\I_3(\O)|$ respectively, as $\O$ ranges over any 
family of orders defined by local conditions.  In both
Theorems~\ref{gensigmaord} and \ref{gensigmaid}, we have seen that 
unless the family consists entirely of maximal orders satisfying a finite number 
of local conditions, these averages depend on the particular families  
of orders over which the averages are taken over. However, we see that 
these two theorems together imply:

\begin{theorem}\label{diff}
Let $(\Sigma_p)$ be any acceptable collection of local specifications
as above, and let $\Sigma$ denote the set
of all isomorphism classes of quadratic orders $\O$  such that 
$\O\otimes\Z_p\in\Sigma_p$ for all $p$.  
  Then, when orders in $\Sigma$ are ordered by their absolute
  discriminants:
\begin{itemize}
\item[{\rm (a)}]
The mean size of 
$|\Cl_3(\O)|-|\I_3(\O)|$ across imaginary quadratic orders $\O$ 
in $\Sigma$ is $1$.
\item[{\rm (b)}]
The mean size of $|\Cl_3(\O)|-\frac13|\I_3(\O)|$ across real 
quadratic orders $\O$ in $\Sigma$ is $1$.
\end{itemize}
\end{theorem}
It is a remarkable fact, which begs for explanation, that the mean
values in Theorem~\ref{diff} do not depend on the family of orders
that one averages over!  In particular, the case of maximal orders
gives Corollary~\ref{maxcase}, because
the only 3-torsion element of the ideal group in a maximal order is the trivial ideal.

We end this introduction by describing the methods used in this paper.  Our approach combines the original methods of Davenport-Heilbronn
with techniques that are class-field-theoretically ``dual'' to those methods, which we explain now. First, recall that Davenport-Heilbronn proved Theorem 1 in~\cite{DH} by: 
\begin{enumerate}
\item[1)] counting appropriate sets of binary cubic forms to compute the number of cubic fields
of bounded discriminant, using a bijection (due to Delone and
Faddeev~\cite{DF}) between irreducible binary cubic forms and cubic
orders; 
\item[2)] applying a duality from class field theory between cubic fields and 3-torsion elements of class groups of
quadratic fields.  
\end{enumerate}

In Sections 2 and 3, we give a new proof of Theorem~1 without class
field theory, by using a direct correspondence between binary cubic
forms and 3-torsion elements of class groups of quadratic orders
proved in~\cite{Bhargava2}, in place of the Delone-Faddeev
correspondence.
We describe a very precise version of this correspondence in Section~2
(cf.~Thm.~\ref{bcfideal}).  In Section~3, we then show how 
the original counting results of Davenport~\cite{Davenport1,Davenport2}---as
applied in the asymptotic count of cubic fields in 
Davenport-Heilbronn~\cite{DH}---can also be used to extract Theorem~1, using the direct 
correspondence between integral binary cubic forms and $3$-torsion elements of class groups of quadratic orders.

To fully illustrate the duality between the original strategy of \cite{DH} and our strategy described above, we give two ``dual'' proofs of Theorem~2. In Section~4, we first generalize the proof of Theorem~1 given in Sections~2 and~3, and then in Section~5, we give a second proof of Theorem~2 via
ring class field theory, generalizing the original proof of Davenport--Heilbronn \cite{DH}.  Both methods involve counting \emph{irreducible} binary cubic forms
in fundamental domains for the action of either $\SL_2(\Z)$ or $\GL_2(\Z)$, as developed in the work
of Davenport~\cite{Davenport1,Davenport2}. However, in our direct method described in Section~4, one must also count points in the ``cusps'' of these fundamental regions!  The
points in the so-called cusp correspond essentially to reducible cubic forms.
We find that reducible cubic forms correspond to 3-torsion elements of
\emph{ideal groups} of quadratic orders (cf.\ Thm.~\ref{reducible}). In
the case of maximal orders, the only torsion element of the ideal
group is the identity,
and thus the points
in the cusps can be ignored 
when proving Theorem~1.  
However, in order to prove Theorems~2 and~3 (which do not restrict to
maximal orders), we must include reducible forms in our counts, and
this is the main goal of Section~4. Isolating the count of 
reducible forms in the fundamental domain for the action of
$\SL_2(\bZ)$ is also what allows us to deduce Theorem~\ref{sigmaid}.

On the other hand, in Section~5, we describe the duality between
nontrivial 3-torsion elements of class groups of a given quadratic
order and cubic fields whose Galois closure is a ring class field of
the fraction field of the quadratic order (cf.\ Prop.~\ref{rcf}).  To then count 3-torsion
elements in the class groups of quadratic orders, we
use the count of cubic fields of bounded discriminant proved by
Davenport--Heilbronn~\cite{DH},
but we allow a
given cubic field to be counted multiple times, as the Galois closure
of a single cubic field can be viewed as the ring class field (of
varying conductor) of multiple quadratic orders (cf.\ \S5.2). This
yields a second proof of Theorem~2; furthermore, it allows us to prove also 
Theorem~3 and Corollary~4, using a generalization of
Davenport and Heilbronn's theorem on the density of discriminants of cubic fields
established in \cite[Thm.~8]{BST}, which counts cubic
orders of bounded discriminant satisfying any acceptable collection of
local conditions.

Finally, in Section~6, we generalize the proof of Theorem~2 given in
Section~3 to general acceptable families of quadratic orders, 
which in combination with Theorem~\ref{gensigmaord} allows
us to deduce Theorems~\ref{gensigmaid} and~\ref{diff}. We note
that, in order to conclude Theorem~\ref{gensigmaid}, we use both of the
``dual'' perspectives provided in the two proofs of Theorem~2.

\section{Parametrization of order 3 ideal classes in quadratic orders}

In this section we recall the parametrization of elements in the
3-torsion subgroups of ideal class groups of quadratic orders in terms
of (orbits of) certain integer-matrix binary cubic forms as proven in
\cite{Bhargava2}. We also deduce various relevant facts that will
allow us to prove Theorems~1 and~2 
in \S3 and \S4, respectively, without using class field theory.

\subsection{Binary cubic forms and 3-torsion elements in class groups}

The key ingredient in the new proofs of Theorems~1 and~2 is a
parametrization of ideal classes of order~3 in quadratic rings by
means of equivalence classes of integer-matrix binary cubic forms,
which was obtained in~\cite{Bhargava2}.  We begin by briefly recalling this
parametrization.

Let $V_\bR$ denote the four-dimensional real vector space of binary cubic
forms $ax^3+bx^2y+cxy^2+dy^3$ where $a,b,c,d\in\R$, and let $V_\Z$ denote
the lattice of those forms for which $a,b,c,d\in\Z$ (i.e., the {\it
  integer-coefficient} binary cubic forms).  The group $\GL_2(\Z)$
acts on $V_\bR$ by the so-called ``twisted action,'' i.e., an element $\gamma \in \GL_2(\Z)$ acts on a binary cubic form $f(x,y)$ by
\begin{equation}
	(\gamma f)(x,y) := \frac{1}{\det(\gamma)}f((x,y)\gamma).
\end{equation}
Furthermore, the action preserves $V_\Z$.
We will be interested in the sublattice of binary
cubic forms of the form $f(x,y)=ax^3+3bx^2y+3cxy^2+dy^3$, called {\it classically
integral} or {\it integer-matrix} if $a,b,c,d$ are integral.  We
denote the lattice of all integer-matrix forms in $V_\bR$ by $V_\Z^\ast$.
Note that $V_\Z^\ast$ has index~9 in $V_\Z$ and is also preserved by $\GL_2(\bZ)$.
We also define the {\it reduced  discriminant} $\disc(\cdot)$ on $V_\Z^\ast$ by
\begin{equation}\label{discdef}
\disc(f) := -\frac{1}{27}\Disc(f) = -3b^2c^2 + 4ac^3 + 4b^3d + a^2d^2 - 6abcd
\end{equation}
where $\Disc(f)$ denotes the usual discriminant of $f$ as an element
of $V_\Z$. It is well-known and easy to check that the action of $\GL_2(\bZ)$ on binary cubic forms preserves (both definitions of) the discriminant.  

In \cite{Eisenstein}, Eisenstein proved a beautiful correspondence
between certain special $\SL_2(\Z)$-classes in $V_\Z^\ast$ and ideal classes
of order 3 in quadratic rings. 
We state here a refinement of Eisenstein's correspondence
obtained in~\cite{Bhargava2}, which gives an
exact interpretation for {\it all} $\SL_2(\Z)$-classes in $V_\Z^\ast$
in terms of ideal classes in quadratic rings.  

To state the theorem, we first require some terminology.  We define a
{\it quadratic ring} over~$\bZ$ (resp.~$\bZ_p$) to be any commutative
ring with unit that is free of rank 2 as a $\Z$-module
(resp. $\bZ_p$-module). An {\it oriented} quadratic ring $\cO$ over
$\bZ$ is then defined to be a quadratic ring along with a specific
choice of isomorphism $\pi: \O/\Z \rightarrow \Z$. Note that an
oriented quadratic ring has no nontrivial automorphisms. Finally, we
say that a quadratic ring (or binary cubic form) is {\it
  nondegenerate} if it has nonzero discriminant.

\begin{theorem}[{\bf \cite[Thm.~13]{Bhargava2}}]\label{bcfideal}  
  There is a natural bijection between the set of nondegenerate
  $\SL_2(\Z)$-orbits on the space $V_\Z^\ast$ of integer-matrix binary
  cubic forms and the set of equivalence classes of triples
  $(\O,I,\delta)$, where $\O$ is a nondegenerate oriented quadratic
  ring over $\bZ$, $I$ is an ideal of $\O$, and $\delta$ is an invertible element
  of $\O\otimes\Q$ such that $I^3\subseteq \delta\cdot \O$ and
  $N(I)^3=N(\delta)$.  $($Here two triples $(\O,I,\delta)$ and
  $(\O',I',\delta')$ are equivalent if there is an isomorphism
  $\phi:\O\rightarrow \O'$ and an element $\kappa\in \O'\otimes\Q$ such
  that $I'=\kappa\phi(I)$ and $\delta'=\kappa^3\phi(\delta)$.$)$  Under
  this bijection, the reduced discriminant of a binary cubic form is equal to
  the discriminant of the corresponding quadratic ring.
\end{theorem}

The proof of this statement can be found in~\cite[\S3.4]{Bhargava2}; here we simply sketch the map. Given a triple $(\O, I, \delta)$, we construct the corresponding binary cubic form as follows: Write $\O = \Z + \Z\tau$ where $\langle 1, \tau\rangle$ is a {\it positively oriented} basis for an oriented quadratic ring, i.e., $\pi(\tau) = 1$. Furthermore, we can write $I = \Z \alpha + \Z \beta$ when $\langle \alpha, \beta \rangle$ is a {\it positively oriented} basis for a $\Z$-submodule of $\cO \otimes \bQ$, i.e., the change-of-basis matrix from the positively-oriented $\langle 1, \tau\rangle$ to $\langle \alpha, \beta \rangle$ has positive determinant. We can then find integers $e_0$, $e_1$, $e_2$, $e_3$, $a$, $b$, $c$, and $d$ satisfying the following equations: 
\begin{equation}\label{bcfdef}
\begin{array}{rcl}
	\alpha^3 &=& \delta(e_0 + a \tau), \\
	\alpha^2\beta &=& \delta(e_1 + b \tau), \\
	\alpha\beta^2 &=& \delta(e_2 + c \tau), \\
	\beta^3 &=& \delta(e_3 + d \tau).
\end{array}
\end{equation}
Then the binary cubic form corresponding to the triple $(\O, I, \delta)$ is $f(x,y) = ax^3 + 3bx^2y + 3cxy^2 + dy^3$.  In basis-free terms, $f$ is the symmetric trilinear form 
\begin{equation}\label{cf}
I\times I\times I\to \Z \qquad \qquad (i_1,i_2,i_3) \mapsto \pi(\delta^{-1}\cdot i_1\cdot i_2\cdot i_3)
\end{equation}
given by applying multiplication in $\O$, dividing by $\delta$, and then applying $\pi$.

On the other hand, given a binary cubic form $f(x,y) = ax^3 + 3bx^2y +
3cxy^2 + dy^3$, we can explicitly construct the corresponding triple
as follows. The ring $\O$ is completely determined by having discriminant equal to $\disc(f)$. Examining the system of equations in (\ref{bcfdef}) shows that a positively oriented basis $\langle \alpha, \beta\rangle$ for $I$  must satisfy
	$$\alpha:\beta = (e_1 + b \tau): (e_2 + c \tau)$$
where 
\begin{equation}
	e_1 = \displaystyle \frac{1}{2}(b^2c - 2ac^2 + abd - \epsilon b), \quad \mbox{and} \quad
	e_2 = \displaystyle -\frac{1}{2}(bc^2 - 2b^2d + acd + \epsilon c).
\end{equation}
Here, $\epsilon = 0$ or $1$ in accordance with whether $\Disc(\cO)
\equiv 0$ or $1$ modulo $4$, respectively. This uniquely determines
$\alpha$ and $\beta$ up to a scalar factor in $\cO \otimes \bQ$, and
once $\alpha$ and $\beta$ are fixed, the system in (\ref{bcfdef})
determines $\delta$ uniquely. The $\cO$-ideal structure of the rank 2
$\bZ$-module $I$ is given by the following action of $\tau$ on the basis
elements of $I$: 
	$$\tau \cdot \alpha = \frac{B + \epsilon}{2} \cdot \alpha + A \cdot \beta \qquad \mbox{and} \qquad \tau \cdot \beta = -C\cdot\alpha + \frac{\epsilon - B}{2} \cdot \beta, \quad \mbox{where}$$
\begin{equation}\label{defABC}
	A = b^2 - ac, \quad B = ad - bc, \quad C = c^2 - bd.
\end{equation}
This completely (and explicitly) determines the triple $(\cO,I,\delta)$ from the binary cubic form $f(x,y)$. Note that the equivalence defined on triples in the statement of the theorem exactly corresponds to $\SL_2(\bZ)$-equivalence on the side of binary cubic forms.

We may also deduce from this discussion a description of the stabilizer in $\SL_2(\Z)$ of an element in $V_\Z^\ast$ in terms of the corresponding triple $(\O,I,\delta)$.  

\begin{corollary}\label{stab}  
The stabilizer in $\SL_2(\Z)$ of a
nondegenerate element $v\in V_\Z^\ast$ is naturally isomorphic to~$U_3(\O_0)$,
where $(\O,I,\delta)$ is the triple corresponding
to $v$ as in Theorem~$\ref{bcfideal}$,
$\O_0=\End_\O(I)$ is the endomorphism ring of~$I$, and
$U_3(\O_0)$ denotes the group of units of $\O_0$ having order
dividing $3$.
\end{corollary}
Indeed, let $v \in V_\bZ^\ast$ be associated to the triple $(\cO, I,
\delta)$ under Theorem~$\ref{bcfideal}$. Then an
$\SL_2(\bZ)$-transformation of the basis $\langle \alpha, \beta
\rangle$ for $I$ preserves the map in (\ref{cf}) precisely when
$\gamma$ acts by multiplication by a cube root of unity in the
endomorphism ring $\O_0$ of $I$.

We may also similarly describe the orbits of $V_\Z^\ast$ under the action of $\GL_2(\Z)$. This simply 
removes the orientation of the corresponding ring $\O$, thus
 identifying the triple $(\O,I,\delta)$ with its quadratic conjugate triple $(\O,\bar I,\bar\delta)$.  

\begin{corollary}\label{gl2bijection}
There is a natural bijection between the set of nondegenerate $\GL_2(\bZ)$-orbits on the space $V_{\bZ}^\ast$ of integer-matrix binary cubic forms and the set of equivalence classes of triples $(\cO, I, \delta)$ where $\cO$ is a nondegenerate $($unoriented$)$ quadratic ring, $I$ is an ideal of $\cO$, and $\delta$ is an invertible element of $\cO \otimes \bQ$ such that $I^3 \subseteq \delta \cdot \cO$ and $N(I)^3 = N(\delta)$. 
 Under this bijection, the reduced discriminant of a binary cubic form is equal to
  the discriminant of the corresponding quadratic ring.
 The stabilizer in $\GL_2(\Z)$ of a
nondegenerate element $v\in V_\Z^\ast$ is given by the 
semidirect product
\[\Aut(\O;I,\delta)\ltimes U_3(\O_0),\] 
where: $(\O,I,\delta)$ is the triple corresponding
to $v$ as in Theorem~$\ref{bcfideal}$; $\Aut(\O;I,\delta)$ 
is defined to be $C_2$ if there exists $\kappa\in(\O\otimes\Q)^\times$ such that 
$\bar I=\kappa I$ and $\bar \delta=\kappa^3\delta$, and 
is defined to be trivial otherwise; 
$\O_0=\End_\O(I)$ is the endomorphism ring of $I$; and
$U_3(\O_0)$ denotes the group of units of $\O_0$ having order
dividing $3$.
\end{corollary}

\begin{proof}
  Given Theorem~\ref{bcfideal}, it remains to check where the
  now-combined $\SL_2(\bZ)$-orbits of an integer-matrix binary cubic
  form $f$ and of $\gamma f$ where $\gamma = \left(\begin{smallmatrix}
      0 & 1 \\1 & 0\end{smallmatrix}\right)$ map to. If the
  $\SL_2(\bZ)$-orbit of $f$ corresponds to a triple $(\cO, I, \delta)$
  under the above bijection, then the $\SL_2(\bZ)$-orbit of $\gamma f$
  corresponds to the triple $(\cO, \bar{I}, \bar{\delta})$ where
  $\bar{\cdot}$ denotes the image under the non-trivial automorphism
  of the unoriented quadratic ring $\cO$. Thus we obtain a
  correspondence between $\GL_2(\bZ)$-orbits of integer-matrix binary
  cubic forms and triples $(\cO, I, \delta)$ as described above except
  that $\cO$ is viewed as a quadratic ring without orientation.

For the stabilizer statement, note that an element $g$ of $\GL_2(\Z)$ preserving $v$ must have determinant either
$+1$ or $-1$.  If $g$ has determinant~1, then when it acts on the basis
$\langle\alpha,\beta\rangle$ of $I$, it preserves the vector
$v=(a,b,c,d)$ in (\ref{bcfdef}) if
and only if $\alpha^3,\alpha^2\beta,\alpha\beta^2,\beta^3$ remain
unchanged; thus $g$ must act by multiplication by a unit $u$ in the unit group
$U(\O_0)$ of $\O_0$ whose cube is 1.
If $g$ has determinant $-1$, then the basis element $\tau$
gets replaced by its conjugate $\bar\tau$ in addition to
$\langle\alpha,\beta\rangle$ being transformed by $g$.  If this is to
preserve the vector $v=(a,b,c,d)$ in (\ref{bcfdef}), then this means 
that conjugation on $\O$ maps $I$ to $\kappa I$ 
and $\delta$ to $\kappa^3 \delta$ for some $\kappa\in(\O\otimes\Q)^\times$. 
The result follows.
\end{proof} 
\begin{remark}\label{rmkzp}{\em 
The statements in
Theorem~\ref{bcfideal}, Corollary~\ref{stab}, and Corollary~\ref{gl2bijection} also hold after base change to~$\bZ_p$, with the same proofs. 
In the case of Theorem~\ref{bcfideal}, in the proof, by a {\it
  positively oriented} basis $\langle\alpha,\beta\rangle$ of an ideal
$I$ of $R$,
we mean that the change-of-basis matrix
from  $\langle 1, \tau\rangle$ to $\langle \alpha, \beta \rangle$ has
determinant equal to a power of $p$; all other details remain identical.
Corollary~\ref{gl2bijection} and its
analogue over $\Z_p$ will be relevant in Section~6, during the proofs of
Theorems~\ref{gensigmaid} and \ref{diff}. 
}\end{remark}

\subsection{Composition of cubic forms and 3-class groups} 

Let us say that an integer-matrix binary cubic form $f$, or its corresponding triple $(\O,I,\delta)$ via the correspondence of Theorem~\ref{bcfideal}, is {\it projective} if $I$ is projective as an
$\O$-module (i.e., if $I$ is invertible as an ideal of $\O$); in such a 
case we have $I^3=(\delta)$.  The bijection of Theorem~\ref{bcfideal} allows us to describe a composition law on the set of projective integer-matrix binary cubic forms, up to $\SL_2(\bZ)$-equivalence, having the same reduced discriminant.  This turns the set of all $\SL_2(\Z)$-equivalence classes of projective integer-matrix binary cubic forms having given reduced discriminant~$D$ into a group,
which is closely related to the group $\Cl_3(\cO)$, if $\cO$ also has discriminant $D$. In this section, we describe this group law and establish some of its relevant properties.

Fix an oriented quadratic ring $\O$. 
Given such an $\O$, we obtain a natural law of composition on 
equivalence classes of triples $(\O,I,\delta)$, where $I$ is an invertible ideal of $\cO$ and $\delta \in (\cO \otimes\bQ)^\times$ such that $I^3 = \delta \cdot \cO$ and $N(I)^3 = N(\delta)$. It is defined by
\[
        (\O,I,\delta)\circ(\O,I',\delta') = (\O,II',\delta\delta').
\]
The equivalence classes of projective triples $(\O,I,\delta)$ thus
form a group under this composition law, which we denote by $H(\O)$
(note that two oriented quadratic rings $\cO$ and $\cO'$ of the same
discriminant are canonically isomorphic, and hence the groups $H(\cO)$ and
$H(\cO')$ are also canonically isomorphic).
By Theorem~\ref{bcfideal}, we also then obtain a corresponding
composition law on $\SL_2(\bZ)$-equivalence classes of integer-matrix
cubic forms $f$ having a given reduced discriminant
$D$ 
(a higher degree analogue of Gauss composition).  We say that such a
binary cubic form $f$ is {\it projective} if the corresponding
$(\cO,I,\delta)$ is projective.  We will sometimes view $H(\cO)$ as the
group consisting of the $\SL_2(\bZ)$-equivalence classes of integer-matrix
binary cubic forms having reduced discriminant equal to $\Disc(\cO)$.

In order to understand the relationship between $H(\cO)$ and $\Cl_3(\cO)$, we first establish a lemma describing the number of preimages
of an ideal class 
under the ``forgetful'' map $H(\cO) \rightarrow \Cl_3(\cO)$ defined by 
$(\O,I,\delta) \mapsto [I]$:

\begin{lemma}\label{deltalemma}
  Let $\O$ be an order in a quadratic field and $I$ an invertible
  ideal of $\O$ whose class has order $3$ in the class group of
  $\O$. Then the number of elements $\delta\in\O$ $($up to cube
  factors in $(\O\otimes\Q)^\times)$ yielding a valid triple
  $(\O,I,\delta)$ in the sense of Theorem~$\ref{bcfideal}$ is $1$ if
  $\Disc(\O)<-3$, and $3$ otherwise.
\end{lemma}

\begin{proof}
  Fix an invertible ideal $I$ of $\O$ that arises in some valid triple. 
The number of elements~$\delta$ 
  having norm equal to $N(I)^3$ and yielding distinct elements of $H(\cO)$ is then $|U^+(\O)/U^+(\O)^{\times3}|$, where
  $U^+(\O)$ denotes the group of units of $\O$ having norm~1. In fact, we have an exact sequence
\begin{equation}\label{hr}
1 \to \frac{U^{+}(\cO)}{U^{+}(\cO)^{\times 3}} \to H(\cO) \to \Cl_3(\cO) \to 1.
\end{equation}
We see that for all orders $\O$ in imaginary quadratic fields other
than the maximal order $\bZ[\sqrt{-3}]$, the unit group has
cardinality 2 or 4, and therefore $|U^+(\O)/U^+(\O)^{\times3}|=1$.
For real quadratic orders $\O$, the unit group has rank one and 
torsion equal to $\{\pm 1\}$, and so $|U^+(\O)/U^+(\O)^{\times3}|=3$. 
Finally, for  $\O=\bZ[\sqrt{-3}]$, we have $|U^+(\O)/U^+(\O)^{\times3}|=3$ as well.
\end{proof}

Equation (\ref{hr}) thus make precise the relationship between
$H(\cO)$ and $\Cl_3(\cO)$. With regard to the sizes of these groups,
we obtain:

\begin{corollary}\label{hr2}
We have $|H(\cO)|= |\Cl_3(\cO)|$ when $\cO$ has discriminant $\Disc(\cO) < -3$, and $|H(\cO)|= 3\cdot|\Cl_3(\cO)|$ otherwise.
\end{corollary}

\subsection{Projective binary cubic forms and invertibility}\label{projsection}

We now wish to explicitly describe the projective binary cubic
forms. Recall that the \emph{quadratic Hessian covariant} of $f(x,y) = ax^3 + 3bx^2y + 3cxy^2 + dy^3$ is given by
$Q(x,y)=Ax^2+Bxy+Cy^2$, where $A$, $B$, $C$ are defined by
(\ref{defABC}); then $Q$ also describes the norm form on $I$ mapping into~$\bZ$.  It is well-known, going back to the work of
Gauss, that $I$ is invertible if and only if $Q(x,y)$ is {\it primitive},
i.e., $(A,B,C)=(b^2-ac,ad-bc,c^2-bd)=1$ (see,
e.g., \cite[Prop.~7.4 \& Thm.~7.7(i)--(ii)]{Cox}). Thus, 
	\begin{equation}\label{projbcf}
	f(x,y)=ax^3+3bx^2y+
3cxy^2+dy^3 \mbox{ is {projective} } \Leftrightarrow
~(b^2-ac,ad-bc,c^2-bd)=1.
	\end{equation}

Let $\mathcal{S}$ denote the set of all projective forms 
$f(x,y)=ax^3+3bx^2y+3cxy^2+dy^3$ in $V_\Z^\ast$.
Let~$V^\ast_{\bZ_p}$ denote the set of all forms
$f(x,y)=ax^3+3bx^2y+3cxy^2+dy^3$ such that $a,b,c,d \in \bZ_p$, and let $\mu_p^\ast(\mathcal{S})$ denote the $p$-adic density of the
$p$-adic closure of $\mathcal{S}$ in $V_{\Z_p}^\ast$, where we normalize the
additive measure $\mu_p^\ast$ on $V_{\Z_p}^\ast = \bZ_p^4$ so that
$\mu_p^\ast(V_{\Z_p}^\ast)=1$.  The following lemma gives the value of~$\mu_p^\ast(\mathcal{S})$:

\begin{lemma}\label{primdensity}
We have $\mu_p^\ast(\mathcal{S})=1-\displaystyle{\frac{1}{p^2}}.$
\end{lemma}

\begin{proof}
  Suppose \begin{equation}\label{primeq}
b^2-ac\,\equiv\, bc-ad\,\equiv\,
  c^2-bd\,\equiv\, 0 \pmod{p}.
\end{equation}
Then the pair $(a,b)$ can take any value except
$(0,r)$, where $r\not\equiv 0$ (mod $p$).  Given any such nonzero
pair $(a,b)$, the variables $c$ and $d$ are then clearly
determined modulo $p$ from $(a,b)$.  If
$(a,b)\equiv(0,0)$~(mod~$p$), then $c$ must also vanish
(mod~$p$), while $d$ can be arbitrary (mod~$p$).  We conclude that
the total number of solutions (mod~$p$) to (\ref{primeq}) for
the quadruple 
$(a,b,c,d)$ is $(p^2-(p-1))+(p-1)=p^2$.  Thus $\mu_p^\ast(\mathcal{S})=
(p^4-p^2)/p^4$, as claimed.
  \end{proof}

\subsection{Reducible forms}

As summarized in the introduction, the correspondence of
Delone-Faddeev in~\cite{DF} between irreducible binary cubic forms and
orders in cubic fields was used by Davenport--Heilbronn~\cite{DH} to
determine the density of discriminants of cubic fields.
Theorem~\ref{bcfideal}, however, gives a different correspondence than
the one due to Delone-Faddeev~\cite{DF}; in particular, it does
\emph{not} restrict to irreducible forms. The question then arises: which
elements of $H(\cO)$ correspond to the integer-matrix binary cubic
forms that are reducible, i.e., that factor over $\Q$ (equivalently,
$\Z$)?  We answer this question here, first by establishing which
triples $(\cO,I,\delta)$ correspond to reducible binary cubic
forms.

\begin{lemma}\label{lemma2}
  Let $f$ be an element of $V^\ast_\Z$, and let
  $(\O,I,\delta)$ be a representative for the corresponding equivalence class of triples as given by
  Theorem~$\ref{bcfideal}$.  Then $f$ has a rational zero as a binary
  cubic form if and only if 
  $\delta$ is a cube in $(\O\otimes\Q)^\times$.
\end{lemma}

\begin{proof}
  Suppose $\delta=\xi^3$ for some invertible $\xi\in \O\otimes\Q$.
  Then by replacing $I$ by $\xi^{-1}I$ and $\delta$ by~$\xi^{-3}\delta$ if necessary, we may assume that $\delta=1$.  Let
  $\alpha$ be the smallest positive element in $I\cap \Z$, and extend
  to a basis $\langle \alpha,\beta\rangle$ of $I$.  Then the binary
  cubic form $f$ corresponding to the basis $\langle
  \alpha,\beta\rangle$ of $I$ via
  Theorem~\ref{bcfideal} 
evidently has a zero, since $\alpha\in\Z$, $\delta=1$, and so 
  $a=0$ in (\ref{bcfdef}).

  Conversely, suppose $(x_0,y_0)\in\Q^2$ with $f(x_0,y_0) = 0$. Without loss of generality, we may assume that $(x_0,y_0) \in \bZ^2$.  If $(\O,I,\delta)$ is the corresponding triple and $I$ has positively oriented basis $\langle \alpha, \beta\rangle$, then by~(\ref{bcfdef}) or (\ref{cf}) we obtain
    $$(x_0 \alpha + y_0 \beta)^3 =  n\delta \quad \mbox{for some  } n \in \bZ.$$
If $\xi = x_0 \alpha + y_0\beta$, then we have $\xi^3 =n\delta$, and taking norms to $\Z$ on both sides reveals that $N(\xi)^3=n^2N(\delta)=n^2N(I)^3$.  Thus $n=m^3$ is a cube.
This then implies  that $\delta$ must be a cube in $(\O\otimes\Q)^\times$ as well,
namely, $\delta=(\xi/m)^3$, as desired.
\end{proof}

The reducible forms thus form a subgroup of $H(\O)$, which we denote
by $H_{\red}(\cO)$; by the previous lemma, it is the subgroup
consisting of those triples $(\O,I,\delta)$, up to equivalence, for
which~$\delta$ is a cube. As in the introduction,
let~$\cI_3(\cO)$ denote the 3-torsion subgroup of the ideal group of
$\cO$, i.e.~the set of invertible ideals $I$ of $\cO$ such that $I^3 =
\cO$. We may then define a map
\begin{equation}
\varphi: \cI_3(\cO) \longrightarrow H(\cO) \qquad \qquad \varphi: I \mapsto (\cO,I,1).
\end{equation}
It is evident that $\im(\cI_3(\cO)) \subseteq H_{\red}(\cO)$. In fact, we show that $\varphi$ defines an isomorphism between $\cI_3(\cO)$ and $H_{\red}(\cO)$: 

\begin{theorem}\label{reducible}
The image of $\cI_3(\cO)$ under $\varphi$ is isomorphic to $H_{\red}(\cO)$.
\end{theorem}
\begin{proof}
  The preimage of the identity $(\cO,\cO,1) \in H(\cO)$ can only
  contain 3-torsion ideals of the form $\kappa\cdot\O$ for $\kappa\in
(\O\otimes \bQ)^\times$. 
To be a 3-torsion ideal, we must have $(\kappa \cO)^3 = \cO$ which implies that $\kappa^3 \in \cO^\times$ and so $\kappa \in \cO^\times$. 
Therefore, the preimage of the identity is simply the ideal $\cO$, and the map is injective. 
It remains to show surjectivity onto $H_{\red}(\cO)$. Assume $(\cO,I,\delta) \in H_{\red}(\cO)$. Since~$\delta$ is a cube by definition, let $\delta = \xi^3$ and recall that $(\cO,I,\delta) \sim (\cO,\xi^{-1}I,1)$. 
Thus $\xi^{-1}I \in \cI_3(\cO)$.
\end{proof} 

\begin{corollary}\label{identity} Assume that $\cO$ is maximal. Then $H_{\red}(\cO)$ contains only the identity element of $H(\cO)$, which can be represented by $(\cO,\cO,1)$. \end{corollary}
\begin{proof} Since maximal orders are Dedekind domains, the only ideal that is 3-torsion in the ideal group is~$\cO$. 
\end{proof}

\section{A proof of Davenport and Heilbronn's theorem on class numbers without class field theory}

Using the direct correspondence of Theorem~\ref{bcfideal}, we can now deduce Theorem~\ref{theoremdh} by counting the relevant binary cubic forms. To do so, we need the following result of Davenport describing the asymptotic behavior of the number of binary cubic forms of bounded reduced discriminant in subsets of $V_\bZ^\ast$ defined by finitely many congruence conditions: 

\begin{theorem}[{\bf \cite{Davenport1}, \cite{Davenport2}, \cite[\S5]{DH}, \cite[Thm.~26]{BST}}]\label{thmdensity}
  Let $S$ denote a set of integer-matrix binary cubic forms 
  in $V_{\bZ}^\ast$ defined by finitely many congruence conditions
  modulo prime powers. Let $V_\bZ^{\ast (0)}$ denote the set of
elements in~$V_{\bZ}^\ast$
  having positive reduced discriminant, and $V_\bZ^{\ast (1)}$ the
  set of elements in~$V_\bZ^\ast$ having reduced negative
  discriminant. For $i = 0$ or $1$, let $N^\ast(S \cap V_\bZ^{\ast
    (i)}, X)$ denote the number of {\em irreducible}
  $\SL_2(\bZ)$-orbits on $S \cap V_\bZ^{\ast (i)}$ having absolute
  reduced discriminant $|\disc|$ less than $X$.  Then
	\begin{equation}\label{ramanujan}
\lim_{X \rightarrow \infty} \frac{N^\ast(S \cap V_\bZ^{\ast (i)},X)}{X} = \frac{\pi^2}{4 \cdot n_i^\ast}\prod_p \mu_p^\ast(S),
	\end{equation}
where $\mu_p^\ast(S)$ denotes the $p$-adic density of $S$ in $V_{\bZ_p}^\ast$, and $n_i^\ast = 1$ or $3$ for $i = 0$ or $1$, respectively. 

\end{theorem}
Note that in both \cite{BST} and \cite{DH}, this theorem is expressed
in terms of $\GL_2(\bZ)$-orbits of binary cubic forms in $V_\bZ$ with
discriminant $\Disc(\cdot)$ defined by $-27\cdot\disc(\cdot)$. Here,
we have stated the theorem for $\SL_2(\bZ)$-orbits of integer-matrix
binary cubic forms, and the $p$-adic measure is normalized so that
$\mu_p^\ast(V^\ast_{\bZ_p})$ = 1. This version is proved in exactly
the same way as the original theorem, but since: 
\begin{enumerate}
\item[(a)] $V_\bZ^\ast$ has index
9 in $V_\bZ$; 
\item[(b)] we use the reduced discriminant $\disc(\cdot)$ instead of $\Disc(\cdot)$; and
\item[(c)] there are two $\SL_2(\bZ)$-orbits in every irreducible
$\GL_2(\bZ)$-orbit,
\end{enumerate} the constant on the right hand side of
(\ref{ramanujan}) changes from $\frac{\pi^2}{12n_i}$ as in \cite{BST}
to $\frac{\pi^2}{4n_i^\ast}$, where $n_i = 6$ or $2$ for $i = 0$ or
$1$, respectively. 

Our goal then is to count the $\SL_2(\bZ)$-orbits of forms in $V_\bZ^{\ast
  (i)}$ that correspond, under the bijection described in
Theorem~\ref{bcfideal}, to equivalence classes of triples $(\cO, I,
\delta)$ where $\cO$ is a maximal quadratic ring and $I$ is
projective. However, if $\cO$ is a maximal 
quadratic ring, then 
all ideals of~$\cO$ are projective,
and so our only restriction on elements $f \in V_\bZ^{\ast (i)}$ then is that
$\disc(f)$ be the discriminant of a maximal quadratic ring.
It is well known that a quadratic
ring $\cO$ is maximal if and only if the odd part of the discriminant of $\cO$ is
squarefree, and $\disc(\cO) \equiv 1$, $5$, $8$, $9$, $12$, or $13
\pmod{16}$. We therefore define for every prime $p$:
	$$\cV_p := \begin{cases} \{ f \in V_{\bZ}^{\ast} : \disc(f) \equiv 1,5,8,9,12,13 \pmod{16}\} & \mbox{ if $p = 2$;} \\
						  \{ f \in V_{\bZ}^{\ast} : \disc_p(f) \mbox{ is squarefree}\} & \mbox{ if $p \neq 2$.} 
			  \end{cases}$$
Here, $\disc_p(f)$ is the $p$-part of $\disc(f)$. If we set $\cV := \cap_p \cV_p$, then $\cV$ is the set of forms in $V_{\bZ}^\ast$ for which the ring $\O$ in the associated triple $(\O,I,\delta)$ is a maximal quadratic ring. The following lemma describes the $p$-adic densities of $\cV$ (here, we are using the fact that the $p$-adic closure of $\cV$ is $\cV_p$):

\begin{lemma}[{\bf\cite[Lem. 4]{DH}}]\label{lemdensity} We have $\mu_p^\ast(\cV_p)= \displaystyle\frac{(p^2 - 1)^2}{p^4}$.
\end{lemma}
We define $N^\ast(\cV \cap V_\bZ^{\ast (i)}, X)$ analogously, as the
number of irreducible orbits in $\cV \cap V_\bZ^{\ast (i)}$ having
absolute reduced discriminant between $0$ and $X$ (for $i =
0,1$). Since we are restricting to irreducible orbits, 
$N^\ast(\cV \cap V_\bZ^{\ast (i)}, X)$ counts
those (equivalence classes of) triples $(\cO,I,\delta)$ where $\cO$ is
maximal with $|\Disc(\cO)| < X$, but by Corollary \ref{identity}, 
the identity of $H(\cO)$ 
is {\it not} included 
in this count.

We cannot immediately apply Theorem~\ref{thmdensity} to compute
$N^\ast(\cV \cap V_\bZ^{\ast (i)}, X)$, as the set $\cV$ is defined by
infinitely many congruence conditions. However, the following
uniformity estimate for the complement of $\cV_p$ for all $p$ will
allow us in \S3.1 
to strengthen (\ref{ramanujan}) to also hold when $S = \cV$:

\begin{proposition}[{\bf\cite[Prop.~1]{DH}}]\label{unifest} Define $\cW_p^\ast = V_{\bZ}^\ast \backslash \cV_p$ for all primes $p$. Then $N(\cW_p^\ast;X) = O(X/p^2)$ where the implied constant is independent of $p$. 
\end{proposition}

\begin{remark}{\em 
None of the proofs of the quoted results in this section use class
field theory except for \cite[Prop.~1]{DH}, which invokes one lemma
(namely, \cite[Lem.~7]{DH}) that is proved in \cite{DH} by class field
theory; however, this lemma immediately follows from our
Thms.~\ref{bcfideal} and~\ref{thmdensity}, which
do not appeal to class field theory.} 
\end{remark}


\subsection{The mean number of 3-torsion elements in the class groups
  of quadratic fields 
without class
field theory (Proof of Theorem~\ref{theoremdh})}
\label{thm1pf}

We now complete the proof of Theorem~\ref{theoremdh}. Suppose $Y$ is any
positive integer. It follows from Theorem~\ref{thmdensity} and Lemma~\ref{lemdensity} that 
	$$\lim_{X \rightarrow \infty} \frac{N^\ast(\cap_{p<Y} \cV_p \cap V_\bZ^{\ast (i)}, X)}{X} = \frac{\pi^2}{4n_i^\ast}\cdot\prod_{p<Y} \left(1-\frac{1}{p^2}\right)^2.$$
Letting $Y$ tend to $\infty$, we obtain 
	$$\limsup_{X \rightarrow \infty} \frac{N^\ast(\cV \cap
          V_\bZ^{\ast (i)}, X)}{X} \leq
        \frac{\pi^2}{4n_i^\ast}\cdot\prod_{p}
        \left(1-\frac{1}{p^2}\right)^2 =
        \frac{3}{2n_i^\ast\zeta(2)}.$$

To obtain a lower bound for $N^\ast(\cV \cap V_\bZ^{\ast (i)}, X)$, we use the fact that
	\begin{equation}\label{above}
	\bigcap_{p < Y} \cV_p \subset (\cV \cup \bigcup_{p \geq Y} \cW_p^\ast).
	\end{equation}
By Proposition~\ref{unifest} and (\ref{above}), we have 
	$$\liminf_{X \rightarrow \infty} \frac{N^\ast(\cV \cap V_\bZ^{\ast (i)}, X)}{X} \geq \frac{\pi^2}{4n_i^\ast}\cdot\prod_{p} \left(1-\frac{1}{p^2}\right)^2 - O(\sum_{p \geq Y} p^{-2}).$$
Letting $Y$ tend to $\infty$ again, we obtain
	$$\liminf_{X \rightarrow \infty} \frac{N^\ast(\cV \cap V_\bZ^{\ast (i)}, X)}{X} \geq \frac{\pi^2}{4n_i^\ast}\cdot\prod_{p} \left(1-\frac{1}{p^2}\right)^2 = \frac{3}{2n_i^\ast\zeta(2)}.$$
Thus,
	$$\lim_{X \rightarrow \infty} \frac{N^\ast(\cV \cap
          V_\bZ^{\ast (i)}, X)}{X} = \frac{9}{n_i^\ast\pi^2}.$$

Finally, we use Corollaries~\ref{hr2} and \ref{identity} to relate $N^\ast(\cV \cap V_\bZ^{\ast (i)}, X)$ and 3-torsion ideal classes in maximal quadratic rings with discriminant less than $X$:
\begin{eqnarray*}
\sum_{{\mbox{\scriptsize $0 <\Disc(\cO) < X$,}}\atop{\mbox{\scriptsize{$\cO$ maximal}}}} \bigl(3\cdot|\Cl_3(\cO)|-1\bigr) &=& 
N^\ast(\cV \cap V_\bZ^{\ast (0)}, X)
 ; \\
\sum_{{\mbox{\scriptsize $0 <-\Disc(\cO) < X$,}}\atop{\mbox{\scriptsize{$\cO$ maximal}}}} \bigl(|\Cl_3(\cO)|-1\bigr) &=& N^\ast(\cV \cap V_\bZ^{\ast (1)}, X).
\end{eqnarray*}
Since \begin{equation}\label{maxordcount}
\displaystyle{\lim_{X\rightarrow\infty}\frac{\displaystyle{\sum_{{\mbox{\scriptsize $0 <\Disc(\cO) < X$,}}\atop{\mbox{\scriptsize{$\cO$ maximal}}}} 1}}{X}} = \displaystyle{\frac{3}{\pi^2}} \qquad \mbox{and} \qquad
\displaystyle{\lim_{X\rightarrow\infty}\frac{\displaystyle{\sum_{{\mbox{\scriptsize $0 <-\Disc(\cO) < X$,}}\atop{\mbox{\scriptsize{$\cO$ maximal}}}} 1}}{X}} 
=  \displaystyle{\frac{3}{\pi^2}},
\end{equation}
we conclude
\[
\begin{array}{rcccl}
\displaystyle{
\lim_{X\rightarrow\infty}\frac{\displaystyle{\sum_{{\mbox{\scriptsize $0 <\Disc(\cO) < X$,}}\atop{\mbox{\scriptsize{$\cO$ maximal}}}} |\Cl_3(\cO)|}}
 {\displaystyle{\sum_{{\mbox{\scriptsize $0 <\Disc(\cO) < X$,}}\atop{\mbox{\scriptsize{$\cO$ maximal}}}} 1}} }
& = &
\displaystyle{\frac{1}{3}\left(1+\lim_{X\rightarrow\infty}\frac{N^\ast(\cV\cap V_\bZ^{\ast (0)};X)}
{{\displaystyle\sum_{{\mbox{\scriptsize $0 <\Disc(\cO) < X$,}}\atop{\mbox{\scriptsize{$\cO$ maximal}}}} 1}}\right)}
& = & 
\displaystyle{\frac{1}{3}\left(1+\frac{9/n_0^\ast}{3}\right)}=
\displaystyle{\frac{4}{3}},\\[.5in]
\displaystyle{
\lim_{X\rightarrow\infty}\frac{\displaystyle \sum_{{\mbox{\scriptsize $0 <-\Disc(\cO) < X$,}}\atop{\mbox{\scriptsize{$\cO$ maximal}}}}|\Cl_3(\cO)|}
 {\displaystyle{\sum_{{\mbox{\scriptsize $0 <-\Disc(\cO) < X$,}}\atop{\mbox{\scriptsize{$\cO$ maximal}}}} 1}} }
&=&
\displaystyle{
1+ \lim_{X\rightarrow\infty}\frac{N^\ast(\cV\cap V_\bZ^{\ast (1)};X)}{\displaystyle \sum_{{\mbox{\scriptsize $0 <-\Disc(\cO) < X$,}}\atop{\mbox{\scriptsize{$\cO$ maximal}}}} 1}}
&=& 
\displaystyle{1+\frac{9/n_1^\ast}{3}} = 
\displaystyle{2}. 
\end{array}
\]

\subsection{Generalization to orders}\label{noncftorders}

The above proof of Theorem~\ref{theoremdh} can be generalized to
orders to yield the special case of Theorem~\ref{diff} where we
average over all quadratic orders. This will also explain why the quantities
being averaged in Theorem~\ref{diff} arise naturally.  All the
ingredients remain the same as in the previous subsection, except that
we now replace $\cV \subset \cV^{\ast}_{\bZ}$ with the set
$\mathcal{S}$ of all {projective} integer-matrix binary cubic forms as
defined in \S2.3. 
Recall that projective forms correspond under
Theorem~\ref{bcfideal} to valid triples with an invertible
ideal. However, since $N^\ast(S,X)$ only counts irreducible orbits, by
Corollary~\ref{hr2} and Theorem~\ref{reducible}, we obtain
\begin{equation}\label{irredcountorders}
N^\ast(\mathcal{S}\cap V_{\bZ}^{\ast (i)}, X) = 
	 \begin{cases} \displaystyle\sum_{0 < \Disc(\cO) < X} 3\cdot|\Cl_3(\cO)| - \sum_{0 < \Disc(\cO) < X} |\cI_3(\cO)| & \mbox{if } i = 0, \\[.35in]
	 \displaystyle\sum_{0 < -\Disc(\cO) < X} |\Cl_3(\cO)| - \sum_{0 < -\Disc(\cO) < X} |\cI_3(\cO)| & \mbox{if } i = 1.
	 \end{cases}
\end{equation}

As before, let $Y$ be any positive integer and let $\cS_p$ denote the
$p$-adic closure of $\cS$ in $V_{\bZ_p}^\ast$, so that $\cap_p \cS_p =
\cS$. It follows from Lemma~\ref{primdensity} and Theorem~\ref{thmdensity} that 
	$$\lim_{X \rightarrow \infty} \frac{N^\ast(\cap_{p<Y} \cS_p \cap V_\bZ^{\ast (i)}, X)}{X} = \frac{\pi^2}{4n_i^\ast}\cdot\prod_{p<Y} \left(1-\frac{1}{p^2}\right).$$
Letting $Y$ tend to $\infty$ gives
	$$\limsup_{X \rightarrow \infty} \frac{N^\ast(\cS \cap V_\bZ^{\ast (i)}, X)}{X} \leq \frac{\pi^2}{4n_i^\ast}\cdot\prod_{p} \left(1-\frac{1}{p^2}\right) = \frac{3}{2n_i^\ast}.$$
Using again $\cW_p^\ast$ to denote $V_\bZ^\ast \backslash \cV_p$, we still have that
	$$\bigcap_{p<Y} \cS_p \subset (\cS \cup \bigcup_{p\geq Y} \cW_p^\ast).$$
Thus, it follows from Theorem~\ref{unifest} that 
	$$\liminf_{X \rightarrow \infty} \frac{N^\ast(\cS \cap V_\bZ^{\ast (i)}, X)}{X} \geq \frac{\pi^2}{4n_i^\ast}\cdot\prod_{p} \left(1-\frac{1}{p^2}\right) - O(\sum_{p \geq Y} p^{-2}),$$
and letting $Y$ tend to $\infty$ gives
	$$\liminf_{X \rightarrow \infty} \frac{N(\cS \cap V_\bZ^{\ast (i)}, X)}{X} \geq \frac{\pi^2}{4n_i^\ast}\cdot\prod_{p} \left(1-\frac{1}{p^2}\right) = \frac{3}{2n_i^\ast}.$$
Thus
	$$\lim_{X\rightarrow\infty} \frac{N^\ast(\cS \cap V_{\bZ}^{\ast (i)},X)}{X} = \frac{3}{2n_i^\ast}.$$
Since
\begin{equation}\label{disc}
\displaystyle{\lim_{X\rightarrow\infty}\frac{\displaystyle \sum_{{0<\Disc(\O)<X}} 1}{X}} 
=  \displaystyle{\frac{1}{2}} \qquad \mbox{and} \qquad
\displaystyle{\lim_{X\rightarrow\infty}\frac{\displaystyle\sum_{{0<-\Disc(\O)<X}} 1}{X}} =  \displaystyle{\frac{1}{2}},
\end{equation}
by (\ref{irredcountorders}) we conclude that
\begin{equation}\label{difforders}
\begin{array}{ccccl}
\displaystyle{
\lim_{X\rightarrow\infty}\frac{\displaystyle{\sum_{{0<\Disc(\O)<X}} |\Cl_3(\cO)| - \frac{1}{3}|\cI_3(\cO)|}}
 {\displaystyle{\sum_{{0<\Disc(\O)<X}} 1}} }
&\!\!=\!\!&
\displaystyle{\frac{1}{3}\left(\frac{\displaystyle \frac{3}{2n_0^\ast}}{\displaystyle\frac{1}{2}}\right)}&\!\!=\!\!&
\displaystyle{1}, \quad \mbox{ and}\\[.5in]
\displaystyle{
\lim_{X\rightarrow\infty}\frac{\displaystyle\sum_{{0<-\Disc(\O)<X}}|\Cl_3(\cO)| - |\cI_3(\cO)|}
 {\displaystyle\sum_{{0<-\Disc(\O)<X}} 1} }
&\!\!=\!\!&
\displaystyle{\frac{\displaystyle\frac{3}{2n_1^\ast}}{\displaystyle\frac{1}{2}}} &\!\!=\!\!& 
\displaystyle{1}.
\end{array}
\end{equation}
This proves Theorem~\ref{diff} 
in the case that $\Sigma$ 
is the set of all isomorphism classes of quadratic
orders.

In the next section, we will count also the reducible
$\SL_2(\bZ)$-orbits of $\cS \cap V_{\bZ}^{\ast (i)}$ having bounded
reduced discriminant, which will establish the mean total number of
3-torsion elements in the class groups of imaginary quadratic
orders and of real quadratic orders, as stated in Theorem~\ref{thmorders}.

\section{The mean number of 3-torsion elements in the ideal 
  groups of quadratic orders (Proofs of Theorems~\ref{thmorders} and~\ref{sigmaid})}

We have seen in \S\ref{noncftorders} that counting irreducible orbits
of integer-matrix binary cubic forms and using the correspondence described in
Theorem~\ref{bcfideal} is not enough to conclude Theorem~2.  In addition, 
Theorem \ref{reducible} shows that in order to establish
Theorem~\ref{sigmaid}, we must compute the number of {\em reducible}
integer-matrix binary cubic forms, up to the action of $\SL_2(\bZ)$,
having bounded reduced discriminant. In \cite{Davenport1,Davenport2},
Davenport computed the number of $\SL_2(\bZ)$-equivalence classes of
irreducible integer-coefficient binary cubic forms of bounded
non-reduced discriminant. In this section, we similarly count
reducible integer-matrix forms with bounded reduced discriminant and
establish the following result, from which both
Theorems~\ref{thmorders} and \ref{sigmaid} follow.

\begin{proposition} \label{redprop}
	Let $h_{\proj,\red}(D)$ denote the number of
        $\SL_2(\bZ)$-equivalence classes of projective and reducible
        integer-matrix binary cubic forms of reduced discriminant $D$. Then
\[
	\begin{array}{ccccl}
	\displaystyle{\sum_{0 < \Disc(\cO) < X}} |H_\red(\cO)| &=& \displaystyle{\sum_{0<D<X}} h_{\proj,\red}(D) &=& \displaystyle{\frac{\zeta(2)}{2\zeta(3)}}\cdot X + o(X) \quad \mbox{and} \\
	\displaystyle{\sum_{0 < -\Disc(\cO) < X}} |H_\red(\cO)| &=& \displaystyle{\sum_{0<-D<X} h_{\proj,\red}(D)} &=& \displaystyle{\frac{\zeta(2)}{2\zeta(3)}}\cdot X + o(X).
	\end{array}
\]
\end{proposition}
Recall that by definition, if $\cO$ is the quadratic ring of discriminant $D$, then $|H_{\red}(\cO)| = h_{\proj,\red}(D).$ 

By Theorem~\ref{reducible}, (\ref{disc}), and 
Proposition~\ref{redprop}, we then obtain:

\begin{corollary}[Theorem~\ref{sigmaid}] Let $\cI_3(\cO)$ denote the $3$-torsion subgroup of the ideal group of the quadratic order $\cO$. Then
	\[
\displaystyle{
\lim_{X\rightarrow\infty}\frac{\displaystyle{\sum_{{0<\Disc(\O)<X}} |\cI_3(\cO)|}}
 {\displaystyle{\sum_{{0<\Disc(\O)<X}} 1}} }
=
\displaystyle{\frac{\zeta(2)}{\zeta(3)}} \qquad \mbox{and} \qquad
\displaystyle{
\lim_{X\rightarrow\infty}\frac{\displaystyle\sum_{{0<-\Disc(\O)<X}}|\cI_3(\cO)|}
 {\displaystyle\sum_{{0<-\Disc(\O)<X}} 1} }
=
\displaystyle{\frac{\zeta(2)}{\zeta(3)}}.
\]
\end{corollary}
Finally, combining Theorem~\ref{sigmaid} with (\ref{difforders}), we conclude: 
\begin{corollary}[Theorem~\ref{thmorders}]
\[
\lim_{X \rightarrow \infty} \frac{\displaystyle\sum_{0<\Disc(\O)<X} |\Cl_3(\cO)|}{\displaystyle\sum_{0<\Disc(\O)<X} 1} 
	= \displaystyle 1 + \frac{1}{3}\cdot\frac{\zeta(2)}{\zeta(3)}, \quad \mbox{and} \quad
\lim_{X \rightarrow \infty} \frac{\displaystyle\sum_{0<-\Disc(\O)<X} |\Cl_3(\cO)|}{\displaystyle\sum_{0<-\Disc(\O)<X} 1} =  \displaystyle 1 + \frac{\zeta(2)}{\zeta(3)}.
\]
\end{corollary}
We now turn to the proof of Proposition~\ref{redprop}.

\subsection{Counting reducible forms of negative reduced discriminant}

We first consider the case of negative reduced discriminant, when the quadratic
Hessian covariant of a binary cubic form is positive-definite.
Gauss described a fundamental domain for
the action of~$\SL_2(\bZ)$ on positive-definite real binary quadratic forms
in terms of inequalities on their coefficients. This allows us to
describe an analogous fundamental domain for the action of~$\SL_2(\bZ)$ on real binary cubic forms
of negative reduced discriminant. Bounding the reduced discriminant
cuts out a region of the fundamental domain which can be described
via suitable bounds on the coefficients of the binary cubic forms  (cf.\ Lemma
\ref{funddomain}). Within this region, we show that the number of
$\SL_2(\bZ)$-classes of reducible integer-matrix cubic forms of
bounded reduced discriminant can be computed, up to a negligible error
term, by counting the number of integer-matrix binary cubic forms $f(x,y)$ in the region whose
$x^3$-coefficient is zero and $x^2y$-coefficient is positive.
We then carry out the latter count explicitly.

\subsubsection{Reduction theory}

Recall that if $f(x,y) = ax^3 + 3bx^2y + 3cxy^2 + dy^3$ is a binary cubic form where $a$, $b$, $c$, $d \in \bZ$, then there is a canonically associated quadratic form $Q$, called the {\em quadratic $($Hessian$)$ covariant} of $f$, with coefficients defined by (\ref{defABC}): 
	\begin{equation}\label{quadcov}
	 Q(x,y) = Ax^2 + Bxy + Cy^2 \quad \mbox{where} \quad A = b^2 - ac, \quad B = ad - bc, \quad \mbox{and} \quad C = c^2 - bd.
	\end{equation}
Note that $\Disc(Q) = \disc(f)$, so if $\disc(f)$ is negative, then its
quadratic covariant is definite. 
The group $\SL_2(\bZ)$ acts on the set of positive-definite real binary quadratic forms, and it is well known that a fundamental domain for this action consists of those quadratic forms whose coefficients satisfy 
	\begin{equation}\label{reduced}
	-A < B \leq A < C \qquad \mbox{or} \qquad 0 \leq B \leq A = C.
	\end{equation}
We call a binary quadratic form whose coefficients satisfy (\ref{reduced}) {\em reduced}. Any binary cubic form of negative reduced discriminant is $\SL_2(\bZ)$-equivalent to one whose quadratic covariant is \emph{reduced}. Furthermore, if two such binary cubic forms are equivalent under $\SL_2(\bZ)$ and both have quadratic covariants that are reduced, then their quadratic covariants are equal. The automorphism group of a reduced quadratic form always includes the identity matrix $\Id_2$ and its negation $-1\cdot\Id_2$. In all but two cases, this is the full automorphism group (the binary quadratic form $x^2 + y^2$ has two more distinct automorphisms while $x^2 + xy + y^2$ has 4 more distinct automorphisms). 

We now describe bounds on the coefficients of a binary cubic form $f$ with reduced quadratic covariant $Q$ satisfying $0 < -\Disc(Q) < X$.
\begin{lemma}[{\bf\cite[Lem.~1]{Davenport1}}]\label{funddomain} Let $a$, $b$, $c$, $d$ be real numbers, and let $A$, $B$, $C$ be defined as in $(\ref{quadcov})$. Suppose that
	\begin{equation}\label{reduced2}
	-A < B \leq A \leq C \qquad \mbox{and} \qquad  0 < 4AC - B^2 < X .
	\end{equation}
Then $$|a| < \frac{\sqrt{2}}{\sqrt[4]{3}}\cdot X^{1/4} \qquad |b| < \frac{\sqrt{2}}{\sqrt[4]{3}}\cdot X^{1/4}$$
	$$|ad| < \frac{2}{\sqrt{3}}\cdot X^{1/2} \qquad |bc| < \frac{2}{\sqrt{3}}\cdot X^{1/2}$$
	$$|ac^3| < \frac{4}{3}\cdot X \qquad |b^3d| < \frac{4}{3}\cdot X$$
	$$|c^2(bc - ad)| < X.$$
\end{lemma}
 	

Note that in the previous lemma, we have included some non-reduced quadratic forms, specifically when $A = C$. However, such cases are negligible by the following lemma:
	
\begin{lemma}[{\bf\cite[Lem.~2]{Davenport1}}] 
The number of integral binary cubic forms satisfying 
	$$-A < B \leq A \leq C \quad \mbox{and} \quad 0 < 4AC - B^2 < X$$ such that $A = C$ is $O(X^{\frac{3}{4}}\log X).$
\end{lemma} 


Finally, the following lemma implies that the number of reducible integer-matrix binary cubic forms with reduced quadratic covariant and bounded reduced discriminant is asymptotically the same as the number of binary cubic forms with $a = 0$, reduced quadratic covariant, and bounded reduced discriminant.

\begin{lemma}[{\bf\cite[Lem.~3]{Davenport1}}] The number of reducible integral binary cubic forms $f$ with $a \neq 0$ that satisfy $-A < B \leq A \leq C$ and for which $0 < -\Disc(Q) < X$ is $O(X^{\frac{3}{4} + \epsilon})$, for any $\epsilon > 0$.
\end{lemma}

Let $h(D)$ denote the number of $\SL_2(\bZ)$-classes of integer-matrix binary cubic forms of reduced discriminant~$D$, and define~$h'(D)$ to be the number of $\SL_2(\bZ)$-classes of integer-matrix binary cubic forms of reduced discriminant~$D$ having a representative with $a = 0$ and quadratic covariant that satisfies $-A < B \leq A \leq C$. Then by the previous two lemmas, we see that 
	\begin{equation}\label{reduction}
	 \sum_{0<-D<X} h(D) = \sum_{0<-D<X} h'(D) + O(X^{\frac{3}{4} + \epsilon}).
	\end{equation}
Thus, we focus our attention on computing $\sum_{0<-D<X} h'(D).$

\subsubsection{The number of binary cubic forms of bounded reduced discriminant with $a = 0$, $b>0$, and reduced quadratic covariant}

If $f(x,y) = 3bx^2y + 3cxy^2 + dy^3$, then the coefficients of the quadratic covariant of $f$ are given by
  $$A = b^2, \quad B = -bc, \quad \mbox{and} \quad C = c^2 - bd,$$
and furthermore $\disc(f) = \Disc(Q) = -3b^2c^2 + 4b^3d$.  We are interested in the number of $\SL_2(\bZ)$-equivalence classes of integer-matrix binary cubic forms with $a = 0$ such that 
  \begin{equation}\label{reduced3}
   -A < B \leq A \leq C \quad \mbox{and} \quad 0 < -\Disc(Q) < X. 
  \end{equation} 
Note that in order for $\Disc(Q)$ to be nonzero, we must have $b \neq 0$. Furthermore, the $\SL_2(\bZ)$-element $-\Id_2$ acts on a form $f(x,y)$ by negating its coefficients, and thus we can assume that our choice of representative for a given $\SL_2(\bZ)$-equivalence class has both $a = 0$ and nonnegative $b$. Apart from the cases when $A = B = C$ or $A = C$ and $B = 0$, the restrictions $a = 0$ and $b > 0$ describe a unique representative in each $\SL_2(\bZ)$-equivalence class of forms satisfying (\ref{reduced3}) and $a = 0$. If $A = B = C$, then the binary cubic form is of the form $3bx^2y - 3bxy^2$. Similarly, if $A = C$ and $B = 0$, then a binary cubic form with such a quadratic covariant is of the form $3bx^2y - by^3$. Thus, by Lemma~\ref{funddomain}, there are $O(X^{1/4})$ such forms in the region described by (\ref{reduced2}) with $a = 0$. If we define $h_1'(D)$ to be the number of integer-matrix binary cubic forms of reduced discriminant $D$ with $a = 0$, $b > 0$, and whose quadratic covariant satisfies $-A < B \leq A \leq C$, then by (\ref{reduction}) we also have
  \begin{equation}\label{reduction2}  
  \sum_{0<-D<X} h(D) = \sum_{0<-D<X} h_1'(D) + O(X^{\frac{3}{4} + \epsilon}).
  \end{equation}

To compute $\sum_{0<-D<X} h_1'(D)$, note that the inequalities in (\ref{reduced3}) imply that $-b^2 < bc < b^2 \leq c^2 - 3bd$ and $0 < 3b^2c^2 - 4b^3d <X$ when $a = 0$; hence if $b > 0$, then
  $$-b < c \leq b \quad \mbox{and} \quad d < \frac{3}{4}\cdot b.$$
Also, since $B^2 \leq AC$, we have $bd \leq 0$, so $d \leq 0$. Using the upper bound on the reduced discriminant of $f$ and the inequality $A \leq C$ from (\ref{reduced3}), we conclude that
  $$\frac{3c^2}{4b} - \frac{X}{4b^3} < d \leq \frac{c^2}{b} - b .$$
The number of integer-matrix binary cubic forms with $a = 0$ and $b > 0$ satisfying (\ref{reduced3}) is therefore
\begin{eqnarray*}
\sum_{0<-D<X} h_1'(D) &=&\sum_{0<b<\frac{\sqrt{2}}{\sqrt[4]{3}}X^{1/4}}\, \sum_{-b < c\phantom{\frac{\sqrt{2}}{\sqrt[4]{3}}}\hspace{-.18in}\leq b}\#\{d \in \bZ : \frac{3c^2}{4b} - \frac{X}{4b^3} < d \leq \frac{c^2}{b} - b\} \\
&=&\sum_{0<b<\frac{\sqrt{2}}{\sqrt[4]{3}}X^{1/4}}\,\sum_{-b < c\phantom{\frac{\sqrt{2}}{\sqrt[4]{3}}}\hspace{-.18in}\leq b} \left(\left(\frac{c^2}{b} -  b\right) - \left(\frac{3c^2}{4b}-\frac{X}{4b^3}\right) + O(1)\right) \\
			&=&\sum_{0<b<\frac{\sqrt{2}}{\sqrt[4]{3}}X^{1/4}} \left(2b\cdot\frac{X}{4b^3} + O(b^2)\right) \\
			&=& \frac{\zeta(2)}{2}X + O(X^{3/4})
\end{eqnarray*}
Thus, by (\ref{reduction2}) the number of $\SL_2(\bZ)$-equivalence classes of reducible integer-matrix binary cubic forms having bounded negative reduced discriminant is given by
  $$\sum_{0<-D<X} h(D) = \frac{\zeta(2)}{2}\cdot X + O(X^{\frac34+\epsilon}).$$

\pagebreak
\subsubsection{Restriction to projective forms}

We now complete the proof of Proposition \ref{redprop} in the case of negative reduced
discriminant by further restricting to projective forms. Let
$h_{1,\proj}'(D)$ be the number of projective integer-matrix binary cubic
forms of reduced discriminant $D$ with $a = 0$, $b > 0$, and reduced
quadratic covariant. By~(\ref{projbcf}), we know that such a form is
projective if and only if $(b^2,bc,c^2-bd) = 1$, or equivalently if
and only if $(b,c) = 1$. Thus $h_{1,\proj}'(D)$ counts those integer-matrix binary cubic forms having reduced
discriminant~$D$, $a = 0$, $b > 0$, $(b,c) = 1$, and reduced
quadratic covariant. Define $h_{1,n}'(D)$ to be the number of integer-matrix binary
cubic forms of reduced discriminant~$D$ with $a = 0$, $b > 0$, $n \mid (b,c)$,
and reduced quadratic covariant. Note that $h_{1,1}'(D) = h_{1,\proj}'(D).$ We
compute $\sum_{0<-D<X} h_{1,\proj}'(D)$ by using the
inclusion-exclusion principle:
	$$\sum_{0<-D<X} h_{1,\proj}'(D) = \sum_{0<-D<X} \sum_{n = 1}^\infty \mu(n)h_{1,n}'(D) = \sum_{n=1}^\infty \mu(n)\cdot\left( \sum_{0<-D<X} h_{1,n}'(D)\right),$$
where $\mu(\cdot)$ denotes the M\"obius function.
	
Fix $n \in \bZ$, and let $3bx^2y + 3cxy^2 + dy^3$ have reduced discriminant $D = -3b^2c^2 + 4b^3d$, $\,b > 0$, and $n \mid (b,c)$. Let $b = n\cdot b_1$ and $c = n \cdot c_1$. Assume that $A = b^2$, $B = -bc$, $C = c^2-bd$ satisfy (\ref{reduced3}). Then
	$$-b_1 < c \leq b_1 \quad \mbox{ and } \quad d < \frac{3}{4}nb_1.$$  Furthermore, $d \leq 0$ and $d$ satisfies
	$$\frac{3nc_1^2}{4b_1} - \frac{X}{4n^3b_1^3} < d \leq \frac{nc_1^2}{b_1} - nb_1.$$
Therefore, the number of integer-matrix binary cubic forms with $a = 0$, $b > 0$, and $n \mid (b,c)$ satisfying~(\ref{reduced3}) is:
\begin{eqnarray*}
	\sum_{0<-D<X} h_{1,n}'(D) 
			&=& \sum_{0<b_1<\frac{\sqrt{2}}{\sqrt[4]{3}n}X^{1/4}}\, \sum_{-b_1 < c_1\phantom{\frac{\sqrt{2}}{\sqrt[4]{3}}}\hspace{-.18in}\leq b_1} \#\{d : \frac{3nc_1^2}{4b_1} - \frac{X}{4n^3b_1^3} < d \leq \frac{nc_1^2}{b_1} - nb_1\} \\
			&=& \sum_{0<b_1<\frac{\sqrt{2}}{\sqrt[4]{3}n}X^{1/4}}\, \sum_{-b_1 < c_1\phantom{\frac{\sqrt{2}}{\sqrt[4]{3}}}\hspace{-.18in}\leq b_1} \left(\left(\frac{nc_1^2}{b_1} -  nb_1\right) - \left(\frac{3nc_1^2}{4b_1}-\frac{X}{4n^3b_1^3}\right) + O(1)\right) \\
			&=& \sum_{0<b_1<\frac{\sqrt{2}}{\sqrt[4]{3}n}X^{1/4}} \left(2b_1\cdot\frac{X}{4n^3b_1^3} + O(nb_1^2)\right) \\
			&=& \frac{\zeta(2)}{2n^3}X + O\left(\frac{X^{3/4}}{n^2}\right),
\end{eqnarray*}	
where the implied constants are independent of $n$.
We conclude that
	\begin{eqnarray*}
	\sum_{0<-D<X} h_{1,\proj}'(D) &=& \sum_{n=1}^\infty \mu(n) \cdot \left(\frac{\zeta(2)}{2n^3}X + O\left(\frac{X^{3/4}}{n^2}\right)\right) \\
	&=& \frac{\zeta(2)}{2\zeta(3)}\cdot X + O(X^{3/4}).
	\end{eqnarray*}.

If we now let $h_{\proj,\red}(D)$ denote the number of $\SL_2(\bZ)$-equivalence classes of projective reducible integer-matrix cubic forms of reduced discriminant $D$, then by the analogous reduction formula as in (\ref{reduction2}), we obtain
	$$\sum_{0<-D<X} h_{\proj,\red}(D) = \frac{\zeta(2)}{2\zeta(3)}\cdot X + o(X).$$

\subsection{Counting reducible forms of positive reduced discriminant}

Recall that implicit in our study of reducible binary cubic forms of
negative reduced discriminant was the fact that their quadratic covariants
were definite, and thus the fundamental domain for positive definite
quadratic forms allowed us to make a well-defined choice for a
representative for each $\SL_2(\bZ)$-class of binary cubic forms we
were counting. If $f(x,y) = ax^3 + 3bx^2y + 3cxy^2 + dy^3$ has
positive reduced discriminant, then its quadratic covariant as defined in
(\ref{quadcov}) is indefinite. However, if we can associate a
different $\SL_2(\bZ)$-covariant quadratic form that is positive
definite to each binary cubic form of positive reduced discriminant, then we can
carry out the analogous count. Again, we follow Davenport
\cite{Davenport2} and note that a binary cubic form of the form
$f(x,y) = ax^3 + 3bx^2y + 3cxy^2 + dy^3$ with positive reduced discriminant
has one real root and two complex roots. Thus, if $\alpha$ denotes the
real root, then we can write
  \begin{equation*}
  f(x,y) = (y-\alpha\cdot x)(Px^2 + Qxy + Ry^2) \quad \mbox{where} \quad
    P = 3b + 3c\alpha + d\alpha^2, \quad Q = 3c + d\alpha, \quad R = d.
    \end{equation*}
We call the binary quadratic form with coefficients $P$, $Q$, and $R$ the \emph{$($definite$)$ quadratic factor} of the binary cubic form $f$. 

\subsubsection{Reduction theory}
As in the case of reduced negative discriminant, a fundamental domain for the action of $\SL_2(\bZ)$ consists of those real quadratic forms $Px^2 + Qxy + Ry^2$ whose coefficients satisfy 
  \begin{equation}\label{reduced4}
  -P < Q \leq P < R \quad \mbox{or} \quad 0 \leq Q \leq P = R.
  \end{equation}
It is clear that any real binary cubic form having positive reduced discriminant is properly equivalent to one with quadratic factor satisfying the inequalities in (\ref{reduced4}). If there are two such binary cubic forms that are equivalent under $\SL_2(\bZ)$ and both quadratic factors satisfy (\ref{reduced4}), then the element of $\SL_2(\bZ)$ taking one cubic form to another must preserve the quadratic factor up to scaling. Thus, it must be an element of the automorphism group of the quadratic factor, hence either $\Id_2$ or $-\Id_2$ when the quadratic factor is not a scalar multiple of $x^2 + y^2$ or $x^2 + xy + y^2$. Apart from these two exceptional cases, in each such $\SL_2(\bZ)$-equivalence class there is one binary cubic form with reduced quadratic factor and $b > 0$. 
Furthermore, using the fact that the non-reduced discriminant of a binary form is the product of the pairwise differences of the roots, one can show that 
	$$\disc(f) = \frac{1}{27}(4PR - Q^2)(P + Q\alpha + R\alpha^2)^2$$
if $\alpha$, $P$, $Q$, $R$ are defined as above. We now state the analogues of Lemmas~\ref{funddomain},~27, and~28.
\pagebreak

\begin{lemma}\label{funddomain2} {\bf \cite[Lem.~1]{Davenport2}} Let $\alpha$, $P$, $Q$, $R$ be real numbers satisfying 
	\begin{equation}\label{reduced5}
	-P < Q \leq P \leq R \quad \mbox{ and } \quad 0 < \frac{1}{27}(4PR-Q^2)(P + Q\alpha + R\alpha^2)^2 <X.
	\end{equation}
If $a$, $b$, $c$, and $d$ are given by the formulas
	$$a = -P\alpha, \quad b = \frac{P-Q\alpha}{3}, \quad c = \frac{Q - R\alpha}{3}, \quad d = R,$$
then
	$$ a < \sqrt{6} X^{1/4} \quad |b| < 2\sqrt[4]{\frac{2}{9}}X^{1/4}$$
	$$|ad| < 3\sqrt{2} X^{1/2} \quad |bc| < \frac{4\sqrt{2}}{3}X^{1/2}$$
	$$|ac^3| < \frac{20}{3}X \quad |b^3d| < \frac{20}{3}X$$
	$$c^2|9bc-ad| < 432X.$$
\end{lemma}


\begin{lemma}[{\bf\cite[Lem.~2]{Davenport2}}] The number of integral binary cubic forms $f$ satisfying
	$$-P < Q \leq P \leq R \quad \mbox{and} \quad 0 < \disc(f) < X$$
such that $P = R$ is $O(X^{\frac{3}{4}}\log X)$.
\end{lemma}

\begin{lemma}[{\bf\cite[Lem.~3]{Davenport2}}] The number of reducible integral binary cubic forms $f$ with $a \neq 0$ that satisfy $-P < Q \leq P \leq R$ and for which $0 < \disc(f) < X$ is $O(X^{\frac{3}{4} + \epsilon})$, for any $\epsilon > 0$.
\end{lemma}

Define $h'(D)$ to be the number of $\SL_2(\bZ)$-classes of integer-matrix binary cubic forms having reduced discriminant $D$ with $a = 0$ and whose quadratic factor satisfies $$-P < Q \leq P \leq R.$$ Then by the previous two lemmas, we see that
	\begin{equation}\label{reduction3}
	\sum_{0<D<X} h(D) = \sum_{0<D<X} h'(D) + O(X^{\frac{3}{4}+\epsilon}).
	\end{equation}
Thus, we focus our attention on computing $\sum_{0<D<X} h'(D)$.

\subsubsection{The number of binary cubic forms of bounded reduced discriminant with $a = 0$, $b > 0$ and reduced quadratic factor}

If $f(x,y) = 3bx^2y + 3cxy^2 + dy^3$, then the coefficients of its quadratic factor are given by
	$$P = 3b, \quad Q = 3c, \quad R = d.$$
Furthermore, $\disc(f) = -\frac{1}{27}\Disc(Px^2 + Qxy + Ry^2)P^2 = -3b^2c^2 + 4b^3d$. We are interested in the number of $\SL_2(\bZ)$-equivalence classes of integer-matrix binary cubic forms $f$ with $a = 0$ such that 
	$$-P < Q \leq P \leq R \quad \mbox{and} \quad 0 < \disc(f) < X.$$
Note that in order for the discriminant of $f$ to be nonzero, we must have $b \neq 0$. Thus, we can assume that our choice of representative for a given $\SL_2(\bZ)$-equivalence class has both $a = 0$ and positive $b$. Apart from the cases when $P = Q = R$ or $P = R$ and $Q = 0$, the restrictions $a = 0$ and $b > 0$ describe a unique representative in each $\SL_2(\bZ)$-equivalence class of forms satisfying (\ref{reduced4}) and~$a = 0$. If $P =  Q = R$, then the binary cubic form  is of the form $3bx^2y + 3bxy^2 + 3by^3$. Similarly, if $P = R$ and $Q = 0$, then the binary cubic form is of the form $3bx^2y + 3by^3$. Thus, by Lemma~\ref{funddomain2}, there are $O(X^{1/4})$ such forms in the region described by (\ref{reduced5}) with $a = 0$. If we define $h_1'(D)$ to be the number of integer-matrix binary cubic forms having reduced discriminant $D$ with $a = 0$ and $b > 0$ and whose quadratic factor satisfies $-P < Q \leq P \leq R$, then by (\ref{reduction3}) we have
	\begin{equation}\label{reduction4}
	\sum_{0<D<X} h(D) = \sum_{0<D<X} h_1'(D) + O(X^{\frac{3}{4} + \epsilon}).
	\end{equation}
	To compute $\sum_{0<D<X} h_1'(D)$, note that the inequalities in (\ref{reduced5}) imply that $-3b < 3c \leq 3b \leq d$ and $0 < -3b^2c^2 + 4b^3d < X$ when $a = 0$; hence if $b > 0$, then
		$$-b < c \leq b \quad \mbox{and} \quad 3b < d.$$
Thus $d > 0$. Using the upper bound on the reduced discriminant of $f$, we conclude that
	 $$3b < d < \frac{X}{4b^3} + \frac{3c^2}{4b}.$$
Therefore, the number of integer-matrix binary cubic forms with $a = 0$ and $b > 0$ satisfying (\ref{reduced5}) is
\begin{eqnarray*}
\sum_{0<D<X} h_1'(D) &=& \sum_{0<b<\sqrt[4]{\frac{32}{9}}X^{1/4}}\, \sum_{-b < c\phantom{\frac{\sqrt{2}}{\sqrt[4]{3}}}\hspace{-.18in}\leq b} \#\{d : 3b < d < \frac{X}{4b^3}  + \frac{3c^2}{4b}\} \\
&=& \sum_{0<b<\sqrt[4]{\frac{32}{9}}X^{1/4}}\, \sum_{-b < c\phantom{\frac{\sqrt{2}}{\sqrt[4]{3}}}\hspace{-.18in}\leq b} \left(\frac{X}{4b^3}  + \frac{3c^2}{4b} - 3b + O(1)\right) \\
			&=& \sum_{0<b<\sqrt[4]{\frac{32}{9}}X^{1/4}} \left(2b\cdot\frac{X}{4b^3} + O(b^2)\right) \\
			&=& \frac{\zeta(2)}{2}X + O(X^{3/4}).
\end{eqnarray*}
Hence, by (\ref{reduction4}), the number of $\SL_2(\bZ)$-equivalence classes of reducible integer-matrix binary cubic forms of bounded positive reduced discriminant is given by
	$$\sum_{0<D<X} h(D) = \frac{\zeta(2)}{2}\cdot X + O(X^{3/4 + \epsilon}).$$

\subsubsection{Restriction to projective forms}

We have seen that the number of $\SL_2(\bZ)$-equivalence classes of 
reducible integer-matrix binary cubic forms with positive reduced discriminant
less than $X$ is $\frac{\zeta(2)}{2}X+o(X)$. We complete the proof
of Proposition~\ref{redprop} by further restricting to projective
forms. Let $h_{1,\proj}'(D)$ be the number  of projective integer-matrix
binary cubic forms having reduced discriminant $D$, $a = 0$, $b > 0$, and
reduced definite quadratic factor. By~(\ref{projbcf}), we know that
such a form is projective if and only if $(b^2,bc,c^2-bd) = 1$, or
equivalently if and only if $(b,c) = 1$. Thus, $h_{1,\proj}'(D)$ counts those integer-matrix forms having reduced discriminant $D$,  $a = 0$, $b > 0$, $(b,c) = 1$, and reduced quadratic factor. Define $h_{1,n}'(D)$ to be the number of $\SL_2(\Z)$-classes of integer-matrix binary cubic forms having reduced discriminant $D$, $a = 0$, $b > 0$, $n \mid (b,c)$, and reduced quadratic factor. Then we have $h_{1,1}'(D) = h_{1,\proj}'(D)$. As before, we compute $\sum_{0<D<X} h_{1,\proj}'(D)$ by using the inclusion-exclusion principle:
	$$\sum_{0<D<X} h_{1,\proj}'(D) = \sum_{0<D<X} \sum_{n = 1}^\infty \mu(n)h_{1,n}'(D) = \sum_{n=1}^\infty \mu(n)\cdot\left( \sum_{0<D<X} h_{1,n}'(D)\right).$$
	
Fix $n \in \bZ$, and let $3bx^2y + 3cxy^2 + dy^3$ have reduced discriminant $D = -3b^2c^2 + 4b^3d$, $b > 0$, and $n \mid (b,c)$. Let $b = n\cdot b_1$ and $c = n \cdot c_1$. Assume that $P$, $Q$, $R$ satisfy (\ref{reduced4});  then
	$$-b_1 < c_1 \leq b_1 \quad \mbox{ and } \quad 3nb_1 < d.$$  Furthermore, $d > 0$ and $d$ satisfies
	$$3nb_1 < d < \frac{X}{4n^3b_1^3} - \frac{3nc_1^2}{4b_1}.$$
Therefore, the number of integer-matrix binary cubic forms with $a = 0$, $b > 0$, and $n \mid (b,c)$ satisfying (\ref{reduced4}) is:
\begin{eqnarray*}
	\sum_{0<D<X} h_{1,n}'(D) &=&\sum_{0<b_1<\sqrt[4]{\frac{32}{9n}}X^{1/4}}\, \sum_{-b_1 < c_1\phantom{\sqrt[4]{\frac{32}{9n}}X^{1/4}} \hspace{-.555in}\leq b_1} \#\{d : 3nb_1 < d < \frac{X}{4n^3b_1^3}  + \frac{3nc_1^2}{4b_1}\} \\
&=&\sum_{0<b_1<\sqrt[4]{\frac{32}{9n}}X^{1/4}}\, \sum_{-b_1 < c_1\phantom{\sqrt[4]{\frac{32}{9n}}X^{1/4}}\hspace{-.555in}\leq b_1} \left(\frac{X}{4n^3b_1^3}  + \frac{3nc_1^2}{4b_1} - 3nb_1 + O(1)\right) \\
			&=&\sum_{0<b_1<\sqrt[4]{\frac{32}{9n}}X^{1/4}} \left(2b_1\cdot\frac{X}{4n^3b_1^3} + O(nb_1^2)\right) \\
			&=& \frac{\zeta(2)}{2n^3}X + O\left(\frac{X^{3/4}}{n^2}\right),
\end{eqnarray*}	
where the implied constants are again independent of $n$.  We conclude that
	\begin{eqnarray*}
	\sum_{0<D<X} h_{1,\proj}'(D) &=& \sum_{n=1}^\infty \mu(n) \cdot \left(\frac{\zeta(2)}{2n^3}X + O\left(\frac{X^{3/4}}{n^2}\right)\right) \\
	&=& \frac{\zeta(2)}{2\zeta(3)}X + O(X^{3/4}).
	\end{eqnarray*}

If we let $h_{\proj,\red}(D)$ denote the number of $\SL_2(\bZ)$-equivalence classes of projective reducible integer-matrix binary cubic forms of reduced discriminant $D$, then by the analogous reduction formula as in (\ref{reduction2}), we obtain
	$$\sum_{0<D<X} h_{\proj,\red}(D) = \frac{\zeta(2)}{2\zeta(3)}\cdot X + o(X).$$

\section{The mean number of 3-torsion elements in class groups of quadratic orders via ring class field theory
(Proofs of Theorems~\ref{thmorders} and~\ref{gensigmaord})}

In the previous sections, we have proven
Theorems~\ref{theoremdh},~\ref{thmorders} and~\ref{sigmaid} without
appealing to class field theory. To prove Theorem~\ref{gensigmaord}
and Corollary~\ref{maxcase}, we use a generalization of the
class-field-theory argument originally due to Davenport and
Heilbronn. In particular, we show that the elements of $\Cl_3(\cO)$
for a quadratic order $\O$ can be enumerated via certain non-Galois cubic
fields.  This  involves the theory of ring class fields (see \cite[\S9]{Cox}), together with the theorem of Davenport and Heilbronn on the density
of discriminants of cubic fields:

\begin{theorem}[\cite{DH}]\label{cubics} Let $N_3(\xi,\eta)$ denote the number of cubic fields $K$ up to isomorphism that satisfy $\xi < \Disc(K) < \eta$. Then 
 $$N_3(0,X) = \frac{1}{12\zeta(3)}X + o(X) \quad \mbox{and} \quad N_3(-X,0) = \frac{1}{4\zeta(3)}X + o(X).$$
\end{theorem}
Note that using class field theory, Davenport and Heilbronn were able to conclude Theorem~\ref{theoremdh} using this count. The new contribution of this section is to extend their argument to all orders and acceptable sets of orders (cf.\ Theorems~\ref{thmorders} and~\ref{gensigmaord}). 

\subsection{Ring class fields associated to quadratic orders} 
For a fixed quadratic order $\cO$, denote the field $\cO \otimes \bQ$ by $k$, and let $\cO_k$ denote the maximal order in $k$. If $[\cO_k:\cO] = f$, then we say that the \emph{conductor} of $\cO$ is equal to $f$ (or sometimes, the ideal generated by $f$ in $\cO_k$).

We begin with a well-known description of $\Cl(\O)$ in terms of ideal classes of~$\O_k$:

\begin{lemma}[{\bf \cite[Prop.\ 7.22]{Cox}}]\label{coxlemma} Let $I_{k}(f)$ denote the subgroup of the group of invertible ideals of~$\cO_k$ consisting of ideals that are prime to $f$, and let $P_{k,\bZ}(f)$ denote the subgroup of the group of principal ideals of $\cO_k$ consisting of those $(\alpha)$ such that $\alpha \equiv a \pmod{f\cO_k}$ for some integer $a$ that is coprime to $f$. Then
	$$\Cl(\cO) \cong I_k(f)/P_{k,\bZ}(f).$$
\end{lemma}

Recall that the {\it ray class group} of $k$ of conductor $f$ is defined as the quotient $\Cl_k(f) := I_k(f)/P_{k,1}(f)$ where $P_{k,1}(f)$ is the subgroup of principal ideals of $\cO_k$ consisting of those $(\alpha)$ such that $\alpha \equiv 1 \pmod  {f\cO_k}$. 
By Lemma~\ref{coxlemma}, we have the following exact sequence:
	\begin{equation}\label{clses}
	1 \rightarrow P_{k,\bZ}(f)/P_{k,1}(f) \rightarrow \Cl_k(f) \rightarrow \Cl(\cO) \rightarrow 1.
	\end{equation}

Let $\sigma$ denote the nontrivial automorphism of $\Gal(k/\bQ)$. For a finite group $G$, let $G[3]$ denote its $3$-Sylow subgroup, and if $G$ is a finite $\Gal(k/\bQ)$-module, then we can decompose $G[3] = G[3]^+ \oplus G[3]^-$ where $G[3]^{\pm} := \{ g \in G: \sigma(g) = g^{\pm 1}\}$.

\begin{lemma}[{\bf \cite[Lem.~1.10]{Nakagawa}}]
If $\cO$ is an order of conductor $f$, $k$ is the quadratic field $\cO \otimes \bQ$, and $\Cl_k(f)$ is the ray class group of $k$ of conductor $f$, then $\Cl_k(f)[3]^- \cong \Cl(\cO)[3]$.
\end{lemma}
\begin{proof}
It is clear that the exact sequence in (\ref{clses}) is a sequence of
finite $\Gal(k/\bQ)$-modules, and implies the exactness of the
following sequences: 
	$$1 \rightarrow \left(P_{k,\bZ}(f)/P_{k,1}(f)\right)[3]^+ \rightarrow \Cl_k(f)[3]^+ \rightarrow \Cl(\cO)[3]^+ \rightarrow 1,$$
	$$1 \rightarrow \left(P_{k,\bZ}(f)/P_{k,1}(f)\right)[3]^- \rightarrow \Cl_k(f)[3]^- \rightarrow \Cl(\cO)[3]^- \rightarrow 1.$$
We can see that $\left(P_{k,\bZ}(f)/P_{k,1}(f)\right)[3]^-$ is trivial since any element $[(\alpha)]$ such that $\alpha \equiv a \pmod{f\cO_k}$ can also be represented by $a\cO_k$. Moreover, for $[a\cO_k] \in \left(P_{k,\bZ}(f)/P_{k,1}(f)\right)[3]^-$, we must have
	$$[a\cO_k] = [\sigma(a\cO_k)] = [a\cO_k]^{-1};$$
hence $[a\cO_k]$ has order dividing~2 and~3, and so must be trivial.
Similarly, $\Cl(\cO)[3]^+$ is trivial since if $[I] \in \Cl(\cO)[3]^+$, then $[I] = [\sigma(I)]$.  Since $N(I) = \sigma(I)I \in \bZ$, $[I]$ has order dividing both~2 and~3 in $\Cl(\cO)$ and is therefore trivial. 
\end{proof}
\begin{proposition}\label{rcf}
Let $\O$ be a quadratic order.  The number of isomorphism classes of cubic fields~$K$ such that
$\Disc(\O)=c^2\Disc(K)$ for some integer $c$ is $\bigl(|\Cl_3(\O)|-1\bigr)/2$.  
\end{proposition}
\begin{proof}
  Let $K$ be a non-Galois cubic field. Then the normal closure
  $\widetilde{K}$ of $K$ over $\bQ$ contains a unique quadratic field
  $k$. One checks that the discriminants of $K$ and $k$ satisfy
  $\Disc(K) = \Disc(k)f^2$, where $f$ is the conductor of the cubic
  extension $\widetilde{K}/k$. By class field theory,
  $\widetilde{K}/k$ corresponds to a subgroup $H$ of $\Cl_k(f)[3]$ of
  index $3$. Since $\widetilde{K}/\bQ$ is a Galois extension, $H$ is a
  $\Gal(k/\bQ)$-module.  If~$\sigma$ denotes the nontrivial
  automorphism in $\Gal(k/\bQ)$, we see that $\widetilde{K}$ is Galois
  over $\bQ$ if and only if $\sigma(\widetilde{K}) =
  \widetilde{K}$. Artin reciprocity implies that the subgroup of
  $\Cl_k(f)[3]$ corresponding to $\sigma(\widetilde{K})$ is the image
  of $H$ under the action of $\sigma$ on $\Cl_k(f)[3]$. Thus, since
  $\widetilde{K}$ is Galois, we conclude that $H$ is stable under
  $\sigma$, and we can write $H = H^+ \oplus H^-$ where $H^\pm := H
  \cap \Cl_k^{\pm}(f)[3]$.

  We now show that $H^+ = \Cl_k^+(f)[3]$. Consider the exact sequence
	$$1 \rightarrow \Gal(\widetilde{K}/k) \rightarrow \Gal(\widetilde{K}/\bQ) \rightarrow \Gal(k/\bQ) \rightarrow 1.$$
        Note that, by definition, $\Gal(\widetilde{K}/k) \cong
        \Cl_k(f)[3]/H$. For any lift of $\sigma$ to $\tilde{\sigma}
        \in \Gal(\widetilde{K}/\bQ)$, Artin reciprocity implies that
        the action of conjugation on $\Gal(\widetilde{K}/k)$ by
        $\tilde{\sigma}$ corresponds to acting by $\sigma$ on
        $\Cl_k(f)[3]/H$. Since $\Gal(\widetilde{K}/\bQ)$ is isomorphic
        to the symmetric group (and is not a direct product),
        conjugation $\Gal(\widetilde{K}/k)$ acts as inversion. Since
        the index of $H$ in $\Cl_k(f)[3]$ is of odd prime order,
        either $H^+ = \Cl_k(f)[3]^+$ or $H^- = \Cl_k(f)[3]^-$.  For
        $\tilde{\sigma}$ to act as inversion on $\Cl_k(f)[3]/H$, we
        must have $\Cl_k(f)[3]/H \cong \Cl_k(f)[3]^{-}/H^-$. By the
        previous lemma, $H$ corresponds to a subgroup of $\Cl(\cO)[3]$
        of index $3$, where $\cO$ is the unique quadratic order of index~$f$ in the ring of integers $\cO_k$.

        Conversely, let $H$ be a subgroup of $\Cl(\cO)[3] \cong
        \Cl_k(f)[3]^-$ of index $3$ where $\cO$ has index~$f$ in $\cO_k$. Then $H$ corresponds to a cubic extension of
        conductor~$d \mid f$, and the action of $\sigma$ is by
        inversion. The above exact sequence therefore does not split,
        and $\Gal(\widetilde{K}/\bQ) = \Gal(\tilde{K}/k) \rtimes
        \Gal(k/\bQ).$ Using Pontryagin duality, we then have
\begin{equation}\label{pont}
\begin{array}{rcl}
	& & \displaystyle \sum_{d \mid f} \#\{\mbox{non-Galois cubic fields with discriminant equal to } \Disc(k)\cdot d^2\} \\
	&=&  \displaystyle \#\{\mbox{subgroups of $\Cl(\O)$ of index $3$}\} = \frac{1}{2}\left(\left|\Cl_3(\cO)\right| - 1\right),
\end{array}
\end{equation} 
where $\cO$ is the quadratic order of index $f$ in the maximal order
of $k$; equivalently, $\cO$ is the unique order with discriminant
equal to $\Disc(k)\cdot f^2$. Note that conjugate cubic fields are
only counted once, and thus we obtain the desired statement (with $c =
\frac{f}{d}$).
\end{proof}

The integer $c$ corresponding to a cubic field $K$ in
Proposition~\ref{rcf} is called the {\it conductor}~of~$K$ relative to
$\O$.  In particular, we see that the conductor~$c$ of $K$ relative to
$\O$ must divide the conductor~$f$ of $\O$.

\subsection{A second proof of the mean number of 3-torsion elements
  in the class groups of quadratic orders (Proof of Theorem~\ref{thmorders} via class field theory)}

Proposition~\ref{rcf} shows that Theorem 2 may be proved by summing,
over all quadratic orders $\O$ of absolute discriminant less than $X$,
the number of cubic fields $K$ such that $\Disc(\cO)/\Disc(K)$ is a
square.  However, in this sum, a single cubic field $K$ may be counted
a number of times since there are many $\O$ for which
$\Disc(\O)/\Disc(K)$ is a square, and one must control the asymptotic
behavior of this sum as $X\to\infty$.

To accomplish this, we rearrange the sum as a sum over the conductor $f$ of $\O$, and then sum
over $\O$ in the interior of this main sum.  This allows us to use a uniformity estimate for large $f$, yielding the desired asymptotic formulae.  More precisely, for $X$ large and $i = 0,1$, we are interested in evaluating
\begin{equation}\label{ni}
N^{(i)}(X) :=\sum_{0<(-1)^i\Disc(\O)<X} \#\{\mbox{cubic fields $K$ such that $\frac{\Disc(\O)}{\Disc(K)}=c^2$, $c \in \bZ$}\}.
\end{equation}
We rearrange this as a sum over $c$ and subsequently over cubic fields:
\begin{equation}\label{theabove}
\begin{array}{rcl}
	N^{(i)}(X) &=& \displaystyle{\sum_{c=1}^\infty \sum_{0<(-1)^i\Disc(\O)<X} \#\{\mbox{cubic fields $K$ such that $\Disc(K) = \Disc(\cO)/c^2$}\}} \\
				&=& \displaystyle{\sum_{c=1}^\infty \sum_{{\mbox{\scriptsize non-Galois $K$ s.t. }}\atop{\mbox{\scriptsize $0 < (-1)^i \Disc(K) < X/c^2$}}} 1}. \\
\end{array}
\end{equation}
Let $Y$ be an arbitrary positive integer. From (\ref{theabove}) and Theorem~\ref{cubics}, we obtain
\begin{eqnarray*}
N^{(i)}(X) &=& \displaystyle\sum_{c=1}^{Y-1}\frac{1}{2 n_i \zeta(3)}\cdot\frac{X}{c^2} + o(X)  + O\left(\displaystyle\sum_{c=Y}^\infty X/c^2\right)\\
		   &=& \displaystyle\sum_{c=1}^{Y-1} \frac{1}{2n_i \zeta(3)}\cdot \frac{X}{c^2} + o(X) + O(X/Y),
\end{eqnarray*}
where $n_0 = 6$ and $n_1 = 2$ (i.e., $n_i$ is the size of the automorphism group of $\bR^3$ if $i = 0$ and $\bR\otimes \bC$ if $i = 1$).
Thus,
	$$\lim_{X \rightarrow \infty} \frac{N^{(i)}(X)}{X} =  \sum_{c=1}^{Y-1} \frac{1}{2n_i \zeta(3)}\cdot \frac{1}{c^2} + O(1/Y).$$
Letting $Y$ tend to $\infty$, we conclude that
	$$\lim_{X \rightarrow \infty} \frac{N^{(i)}(X)}{X} =  \sum_{c=1}^{\infty} \frac{1}{2n_i \zeta(3)}\cdot\frac{1}{c^2} = \frac{\zeta(2)}{2n_i \zeta(3)}.$$
Finally, using Proposition~\ref{rcf} and (\ref{disc}), we obtain 
	$$\lim_{X \rightarrow \infty} \frac{\displaystyle\sum_{0<(-1)^i\Disc(\O)<X} |\Cl_3(\cO)|}{\displaystyle\sum_{0<(-1)^i\Disc(\O)<X} 1} = 1 + \lim_{X\rightarrow \infty} \frac{4\cdot N^{(i)}(X)}{X} = 
	      \begin{cases} \displaystyle 1 + \frac{\zeta(2)}{3\zeta(3)} & \mbox{if } i = 0, \\[.25in]
	       \displaystyle 1 + \frac{\zeta(2)}{\zeta(3)} & \mbox{if } i = 1.
	      \end{cases}$$

\subsection{The mean number of 3-torsion elements in the class groups
  of quadratic orders in acceptable families (Proof of Theorem~\ref{gensigmaord})}

We now determine the mean number of 3-torsion elements in the class
groups of quadratic orders satisfying any acceptable set of local
conditions. As described in the introduction, for each prime~$p$, let
$\Sigma_p$ be a set of isomorphism classes of nondegenerate quadratic
rings over $\bZ_p$. Recall that a collection $\Sigma = (\Sigma_p)$ is
\emph{acceptable} if for all sufficiently large $p$, the set
$\Sigma_p$ contains the maximal quadratic rings over $\bZ_p$. We
denote by $\Sigma$ the set of quadratic orders $\cO$ over $\bZ$, up to
isomorphism, such that $\cO \otimes \bZ_p \in \Sigma_p$ for all $p$.
For a quadratic order~$\cO$, we write ``$\cO \in \Sigma$'' (or say
that ``$\cO$ is a $\Sigma$-order'') if $\cO \otimes \bZ_p \in
\Sigma_p$ for all primes $p$.

Let us fix an acceptable collection $\Sigma = (\Sigma_p)$ of local specifications. We first recall a necessary generalization of Theorem~\ref{cubics}:

\begin{theorem}[{\bf \cite[Thm.~8]{BST}}]\label{gencubics} Let
  $(\Sigma^{(3)}_p) \cup \Sigma^{(3)}_\infty$ be an acceptable
  collection of local specifications {\em for cubic orders}, i.e., for
  all sufficiently large primes $p$, $\Sigma^{(3)}_p$ contains all
  maximal cubic rings over $\bZ_p$ that are not totally ramified. Let
  $\Sigma^{(3)}$ denote the set of all isomorphism classes of orders
  $\cO_3$ in cubic fields for which $\cO_3 \otimes \bQ_p \in
  \Sigma^{(3)}_p$ for all $p$ and $\cO_3 \otimes \bR \in
  \Sigma^{(3)}_\infty$, and denote by $N_3(\Sigma^{(3)},X)$ the number
  of cubic orders $\cO_3 \in \Sigma^{(3)}$ that satisfy
  $|\Disc(\cO_3)| < X$. Then
  $$N_3(\Sigma^{(3)},X) = \left(\frac{1}{2}\sum_{R_3 \in \Sigma^{(3)}_\infty} \frac{1}{|\Aut(R_3)|}\right)\cdot\prod_p \left(\frac{p-1}{p}\cdot\sum_{R_3 \in \Sigma^{(3)}_p} \frac{1}{\Disc_p(R_3)}\cdot \frac{1}{|\Aut(R_3)|}\right)\cdot X + o(X),$$
 where $\Disc_p(\cdot)$ denotes the $p$-power of $\Disc(\cdot)$.
\end{theorem}

We can use the above theorem to prove Theorem~\ref{gensigmaord} by
comparing the number of 3-torsion elements in the class groups of
quadratic $\Sigma$-orders $\cO$ of absolute discriminant less than $X$
and the number of cubic fields corresponding to such class group
elements of $\cO \in \Sigma$ with absolute discriminant less than $X$
via Proposition~\ref{rcf}. Analogous to (\ref{ni}), we define
  \begin{equation}\label{nisigma}
  \begin{array}{rcl}
  N^{(i)}(X, \Sigma) &:=&\displaystyle \sum_{{\mbox{\scriptsize $\cO \in \Sigma$ s.t.}}\atop{\mbox{\scriptsize $0<(-1)^i\Disc(\O)<X$}}} \#\{\mbox{cubic fields $K$ such that $\frac{\Disc(\O)}{\Disc(K)}=c^2$, $c \in \bZ$}\} \\
	&=& \displaystyle{\sum_{c=1}^\infty \sum_{{\mbox{\scriptsize $\cO \in \Sigma$ s.t.}}\atop{\mbox{\scriptsize $0<(-1)^i\Disc(\O)<X$}}} \#\{\mbox{cubic fields $K$ such that $\Disc(K) = \frac{\Disc(\cO)}{c^2}$}\}}. 
  \end{array}
  \end{equation}
  For any $c \in \bZ$, let $c^{-2}\Sigma$ denote the set of quadratic
  orders $\cO$ that contain an index $c$ $\Sigma$-order.  We can
  decompose $c^{-2}\Sigma$ into the following local specifications:
  for all $p$, if $p^{c_p} \mid\mid c$ where $c_p \in \bZ_{\geq 0}$,
  let $p^{-2c_p}\Sigma_p$ denote the set of nondegenerate quadratic
  rings over $\bZ_p$ which contain an index $p^{c_p}$ subring that
  lies in $\Sigma_p$. It is then clear that $(p^{-2c_p}\Sigma_p) =
  c^{-2}\Sigma$ is acceptable since $\Sigma$
  is. 

  Finally, let $\Sigma^{(3),c}$ be the set of cubic fields $K$ such
  that there exists a quadratic order $\cO \in c^{-2}\Sigma$ with
  $\Disc(K) = \Disc(\cO)$. These are the set of cubic fields $K$ such
  that their {\em quadratic resolvent ring} contains a $\Sigma$-order
  with index $c$, or equivalently, is a $c^{-2}\Sigma$-order. Let
  $D(K)$ denote the quadratic resolvent ring of the cubic field $K$,
  i.e., $D(K)$ is the unique quadratic order with discriminant equal to
  that of $K$. The local specifications for $\Sigma^{(3),c}$ are as
  follows: for all $p$ and with~$c_p$ defined as above,
  $\,\Sigma^{(3),c_p}_p$ is the set of 
  \'etale cubic algebras $K_p$ over $\Q_p$ such
  that the quadratic resolvent ring $D(K_p)$ over $\Z_p$ is a
  $p^{-2c_p}\Sigma_p$-order. Meanwhile, $\Sigma^{(3),c}_\infty$ has
  one cubic ring over $\bR$ specified by the choice $i = 0$ or $1$: it
  contains $\bR^3$ if $i = 0$ and $\bR \otimes \bC$ if $i = 1$. Then
  $\Sigma^{(3),c} = (\Sigma_p^{(3),c_p}) \cup \Sigma^{(3),c}_\infty$,
  and in order to use Theorem~\ref{gencubics}, it remains to show that
  $\Sigma^{(3),c}$ is acceptable.

  To show the acceptability of $\Sigma^{(3),c}$, consider any $p>2$
  large enough so that $\Sigma_p$ contains all maximal quadratic
  rings and $c_p = 0$, i.e., $p \nmid c$. Let $K_p$ be an \'etale cubic
  algebra over $\bQ_p$ that is not totally ramified. This implies that
  $p^2 \nmid \Disc(K_p)$, and so $p^2 \nmid D(K_p)$; therefore,
  $D(K_p)$ must be maximal. By our choice of $p$, we have $D(K_p) \in
  \Sigma_p$, and so $K_p \in \Sigma^{(3),c}$.  Hence $\Sigma^{(3),c}$
  is acceptable.

Using these definitions, we can rewrite $N^{(i)}(X,\Sigma)$ as
  \begin{equation}\label{theabove2}
  \begin{array}{rcl}
  N^{(i)}(X, \Sigma) &=& \displaystyle{\sum_{c=1}^\infty \sum_{{\mbox{\scriptsize $\cO \in \Sigma$ s.t.}}\atop{\mbox{\scriptsize $0<(-1)^i\Disc(\O)<X$}}} \#\{\mbox{cubic fields $K$ such that $\Disc(K) = \frac{\Disc(\cO)}{c^2}$}\}} \\
  			&=& \displaystyle{\sum_{c=1}^\infty \sum_{{\mbox{\scriptsize $\cO \in c^{-2}\Sigma$ s.t.}}\atop{\mbox{\scriptsize $0<(-1)^i\Disc(\O)<\frac{X}{c^2}$}}} \#\{\mbox{cubic fields $K$ such that $\Disc(K) = \Disc(\cO)$}\}}	 \\
			&=& \displaystyle{\sum_{c=1}^\infty \sum_{{\mbox{\scriptsize $K \in \Sigma^{(3),c}$ s.t. }}\atop{\mbox{\scriptsize $0 < (-1)^i \Disc(K) < \frac{X}{c^2}$}}} 1} \\
			&=& \displaystyle{\sum_{c=1}^\infty  ~ ~N_3\left(\Sigma^{(3),c},\frac{X}{c^2}\right)}.
  \end{array}
  \end{equation}
Again, let $Y$ be an arbitrary positive integer. From (\ref{theabove2}), Theorem~\ref{gencubics}, and Theorem~\ref{cubics}, we obtain
\begin{equation*}
  \begin{array}{rcl}
  N^{(i)}(X, \Sigma) &=& \displaystyle{\sum_{c=1}^{Y-1} \frac{1}{2n_i} \cdot \prod_p \left(\frac{p-1}{p}\cdot\sum_{K_p \in \Sigma^{(3),c_p}_p} \frac{1}{\Disc_p(K_p)}\cdot \frac{1}{|\Aut(K_p)|}\right)\cdot \frac{X}{c^2}} \\
  && + \,\,\displaystyle{O\left(\sum_{c=Y}^\infty X/c^2\right)} + \,o(X),
  \end{array}
 \end{equation*}
where $n_0 = 6$ and $n_1 = 2$ as before. Thus, since $O\left(\sum_{c=Y}^\infty X/c^2\right) = O\left(X/Y\right)$, we have
	$$\lim_{X\rightarrow \infty} \frac{N^{(i)}(X,\Sigma)}{X} = \displaystyle{\sum_{c=1}^{Y-1} \frac{1}{2n_i} \cdot \prod_p \left(\frac{p-1}{p}\cdot\sum_{K_p \in \Sigma^{(3),c_p}_p} \frac{1}{\Disc_p(K_p)}\cdot \frac{1}{|\Aut(K_p)|}\right)\cdot \frac{1}{c^2}} + \displaystyle{O\left(1/Y\right)}.$$
Letting $Y$ tend to $\infty$, we conclude that 
	$$\lim_{X\rightarrow \infty} \frac{N^{(i)}(X,\Sigma)}{X} = \displaystyle{\sum_{c=1}^{\infty} \frac{1}{2n_i} \cdot \prod_p \left(\frac{p-1}{p}\cdot\sum_{K_p \in \Sigma^{(3),c_p}_p} \frac{1}{\Disc_p(K_p)}\cdot \frac{1}{|\Aut(K_p)|}\right)\cdot \frac{1}{c^2}}.$$
Let $M_\Sigma$ be defined as in (\ref{massdef}) and let $M^{\eq}_\Sigma$ be the following product of local masses:

\begin{equation}\label{massdef2}
M^{\eq}_{\Sigma} := 
\prod_p\,
\frac{{\displaystyle\sum_{R\in\Sigma_p}\frac 
{C^{\eq}(R)}
{\Disc_p(R)}}}
{{\displaystyle \sum_{R\in\Sigma_p}\frac1{\Disc_p(R)}\cdot\frac{1}{\#\Aut(R)}}} 
= \prod_p\, \frac{{\displaystyle \sum_{R\in \Sigma_p}\frac{C^{\eq}(R)}{\Disc_p(R)}}}{\displaystyle{\sum_{R \in \Sigma_p} {\frac{1}{2\cdot \Disc_p(R)}}}},
\end{equation}
where $C^{\eq}(R)$ is defined for an \'etale quadratic algebra $R$ over $\bZ_p$ as the (weighted) number of \'etale cubic algebras $K_p$ over $\bQ_p$ such that $R = D(K_p)$:
	$$C^{\eq}(R) := \sum_{{\mbox{\scriptsize $K_p$ \'etale cubic $/\bQ_p$}}\atop{\mbox{\scriptsize s.t. $R = D(K_p)$}}} \frac1{\#\Aut(K_p)}.$$ 
Then
\begin{equation}\label{star2}
  \begin{array}{rcl}
 \displaystyle{ \lim_{X\rightarrow \infty} \frac{N^{(i)}(X, \Sigma)}{X}} &=& \displaystyle{\frac{1}{2n_i}\cdot \displaystyle\sum_{c=1}^\infty \frac{1}{c^2} \cdot \prod_p \left(\frac{p-1}{p}\cdot\sum_{K_p \in \Sigma^{(3),c_p}_p} \frac{1}{\Disc_p(K_p)}\cdot \frac{1}{|\Aut(K_p)|}\right)}  \\
  			&=& \displaystyle{\frac{1}{2n_i}\cdot \displaystyle\sum_{c=1}^\infty \frac{1}{c^2} \cdot \prod_p \left(\frac{p-1}{p}\cdot\sum_{R \in p^{-2c_p}\Sigma_p} \frac{1}{\Disc_p(R)} \cdot C^{\eq}(R)\right)}  \\
			&=& \displaystyle{\frac{1}{2n_i}\cdot \displaystyle \prod_p \left(\frac{p-1}{p}\cdot\sum_{i=0}^\infty \sum_{{\mbox{\scriptsize $R \in \Sigma_p$  s.t. }}\atop{\mbox{\scriptsize $\exists R'$ s.t. $ [R':R] = p^{i}$}}} \frac{1}{\Disc_p(R)} \cdot C^{\eq}(R')\right)}  \\
			&=& \displaystyle{\frac{1}{2n_i}\cdot \displaystyle \prod_p \left(\frac{p-1}{p}\cdot\sum_{{R \in \Sigma_p }} \frac{1}{\Disc_p(R)} \cdot C(R)\right)} .
  \end{array}
 \end{equation}	
Recall that $C(R)$ is defined as the (weighted) number of etale cubic algebras $K_p$ over $\bQ_p$ such that $R \subset D(K_p)$ (cf.\ Equation (\ref{Cdef})). The final equality follows from the fact that if we fix $R \in \Sigma_p$, the cubic algebras $K_p$ with discriminant $p^{2i}\cdot\Disc_p(R)$ are disjoint for distinct choices of $i$. (The pentultimate equality follows from unique factorization of integers.)

Using (\ref{pont}) and (\ref{nisigma}), we see that 
	\begin{equation}\label{star}
	2\cdot N^{(i)}(X,\Sigma) = \sum_{{\mbox{\scriptsize $\cO \in \Sigma$ s.t.}}\atop{\mbox{\scriptsize $0<(-1)^i\Disc(\O)<X$}}} (\#\Cl_3(\cO)-1).
	\end{equation}
We now have the following elementary lemma counting quadratic orders:
\begin{lemma}[{\bf \cite[\S4]{Bhamass1}}]\label{dht}
\hfill
\begin{itemize}
\item[$($a$)$] The \,number \,of \,\,real \,\,$\Sigma$-orders\, $\O$ \,with
  \,\,$|\Disc(\O)|<X$ is asymptotically
  $$\frac{1}{2}\cdot\prod_p\Bigl(\frac{p-1}{p}\cdot \sum_{R\in\Sigma}
\frac{1}{\Disc_p(R)}\cdot\frac{1}{\#\Aut(R)}\Bigr)\cdot X.$$ 

\item[$($b$)$] The number of complex $\Sigma$-orders $\O$ with $
|\Disc(\O)|<X$ is asymptotically 
  $$\frac{1}{2} \cdot \prod_p\Bigl(\frac{p-1}{p}\cdot \sum_{R\in\Sigma}
\frac{1}{\Disc_p(R)}\cdot\frac{1}{\#\Aut(R)}\Bigr)\cdot X.$$
\end{itemize}
\end{lemma}
By (\ref{star2}), (\ref{star}), and Lemma~\ref{dht}, we then obtain 
\begin{equation}
\begin{array}{rcl}
\displaystyle{\lim_{X \rightarrow \infty}} \displaystyle{\frac{\displaystyle{\sum_{{\cO \in \Sigma,}\atop{0<(-1)^i\Disc(\O)<X}}} \#\Cl_3(\cO)}{\displaystyle{\sum_{{\cO \in \Sigma,}\atop{0<(-1)^i\Disc(\O)<X}} 1 }}} 
		&=& 1 + \displaystyle{\lim_{X \rightarrow \infty}} \frac{\displaystyle{2\cdot N^{(i)}(X,\Sigma)}}{\displaystyle{\sum_{{\cO \in \Sigma,}\atop{0<(-1)^i\Disc(\O)<X}}}  1 } \\[.35in]
		&=& 1 +  \frac{\displaystyle{\frac{1}{n_i}\cdot \displaystyle \prod_p \left(\frac{p-1}{p}\cdot\sum_{{R \in \Sigma_p }} \frac{1}{\Disc_p(R)} \cdot C(R)\right)}}{\displaystyle{\frac{1}{2} \cdot \prod_p \left(\frac{p-1}{p} \cdot \sum_{R \in \Sigma_p}   \frac{1}{\Disc_p(R)} \cdot \frac{1}{\#\Aut(R)}\right)}} \\[.6in]
		&=& 1 + \displaystyle{\frac{2}{n_i}\cdot \prod_p \frac{\displaystyle{\sum_{{R \in \Sigma_p }} \frac{C(R)}{\Disc_p(R)}}}{\displaystyle{\sum_{R \in \Sigma_p} \frac{1}{2\cdot \Disc_p(R)}}}} \\
		&=& 1 + \displaystyle{\frac{2}{n_i} \cdot M_\Sigma}.
\end{array}
\end{equation}
As $n_0 = 6$ and $n_1 = 2$, this proves Theorem~\ref{gensigmaord}.

\subsection{Families of quadratic fields defined by finitely many
  local conditions always have the
  same average number of 3-torsion elements in their class groups (Proof of Corollary 4)}

We now consider the special case of Theorem~\ref{gensigmaord} where
$(\Sigma_p)$ is any acceptable collection of local specifications of
\emph{maximal} quadratic rings over $\bZ_p$. Then, if $\Sigma$ denotes
the set of all isomorphism classes of quadratic orders $\cO$ such that
$\cO \otimes \bZ_p \in \Sigma_p$ for all $p$, then $\Sigma$ will be a set
of maximal orders satisfying a specified set of local conditions at
some finite set of primes. We prove in this section that regardless of
what acceptable set of maximal orders $\Sigma$ is, the average size of
the 3-torsion subgroup in the class groups of imaginary (resp. real)
quadratic orders in $\Sigma$ is always given by $2$ (resp.\ $\frac{4}{3}$). To do so,
we use Theorem~\ref{gensigmaord} and show that $M_\Sigma = 1$ in these
cases.

\begin{lemma} For any maximal quadratic ring $R$ over $\bZ_p$, we have $C(R) = \frac{1}{2}$, where $C(R)$ denotes the weighted number of \'etale cubic algebras $K_p$ over $\bQ_p$ such that $R$ is contained in the unique quadratic algebra over $\bZ_p$ with the same discriminant as $K_p$ $($cf.\ Equation $(\ref{Cdef}))$.
\end{lemma}

\begin{proof} For all primes $p \neq 2$, there are 4 maximal quadratic rings over $\bZ_p$ (up to isomorphism), namely $\bZ_p \oplus \bZ_p$, $\bZ_p[\sqrt{p}]$, $\bZ_p[\sqrt{\epsilon}]$, and $\bZ_p[\sqrt{\epsilon\cdot p}]$, where $\epsilon$ is an integer which is not a square mod $p$. For each choice of $R$, we compute $C(R)$:
\begin{eqnarray*}
C(\bZ_p \oplus \bZ_p) &=& \frac{1}{\#\Aut(\bQ_p \oplus \bQ_p \oplus \bQ_p)} + \frac{1}{\#\Aut(\bQ_{p^3})} = \frac{1}{6} + \frac{1}{3} = \frac{1}{2},\\
C(\bZ_p[\sqrt{\alpha}]) &=& \frac{1}{\#\Aut(\bQ_p \oplus \bQ_p[\sqrt{\alpha}])} = \frac{1}{2} \quad \mbox{for $\alpha = p$, $\epsilon$  and $p\cdot\epsilon.$}
\end{eqnarray*}
Here, $\bQ_{p^3}$ denotes the unique unramified cubic extension of $\bQ_p$. Note that any ramified cubic field extension $K_p$ of $\bQ_p$ has discriminant divisible by $p^2$ (since $p$ will have ramification index~$3$ in $K_p$). This implies that $D(K_p)$ is not maximal for ramified $K_p$, and so no maximal quadratic ring is contained in $D(K_p)$.

When $p = 2$, there are 8 maximal quadratic rings over $\bZ_2$ (up to isomorphism), namely, $\bZ_2 \oplus \bZ_2$ and $\bZ_2[\sqrt{\alpha}]$ where $\alpha = 2$, $3$, $5$, $6$, $7$, $10$, or $14$. As above, we have that $C(\bZ_2 \oplus \bZ_2) = \frac{1}{2}$. Finally, it is easy to see that for each possible value of $\alpha$,
	$$C(\bZ_2[\sqrt{\alpha}]) = \frac{1}{\#\Aut(\bQ_2 \oplus \bQ_2[\sqrt{\alpha}])} = \frac{1}{2}.$$
Again, any ramified cubic extension $K_2$ of $\bQ_2$ has discriminant divisible by $4$, which implies that $D(K_2)$  does not contain any maximal orders.
\end{proof}

By the above lemma, we see that if $(\Sigma_p)$ is any acceptable collection of local specifications of maximal quadratic rings over $\bZ_p$, then
$$\sum_{R \in \Sigma_p} \frac{C(R)}{\Disc_p(R)} = \sum_{R \in \Sigma_p} \frac{1}{2 \cdot \Disc_p(R)}.$$
Thus $M_\Sigma = 1$, and so by Theorem~\ref{gensigmaord} we obtain Corollary~\ref{maxcase}.

\section{The mean number of 3-torsion elements in the ideal groups of
  quadratic orders in acceptable families (Proof of Theorems~\ref{gensigmaid} and \ref{diff})} 

Finally, we prove Theorems~\ref{gensigmaid} and \ref{diff}, which generalize Theorem \ref{sigmaid} and the work of \S3.2 by determining the mean number of 3-torsion elements
in the ideal groups of quadratic orders satisfying quite general sets
of local conditions.

To this end, fix an acceptable collection $(\Sigma_p)$ of local
specifications for quadratic orders, and fix any $i\in\{0,1\}$.  Let $S = S({\Sigma,i})$
denote the set of all irreducible elements $v\in
V_\Z^{\ast (i)}$ such that, in the corresponding triple $(\O,I,\delta)$, we
have that $\O\in\Sigma$ and $I$ is invertible as an ideal class of $\O$
(implying that $I\otimes\Z_p$ is the trivial ideal class of
$\O\otimes\Z_p$ for all $p$).  

\begin{proposition}[{\bf \cite[Thm.~31]{BST}}] Let $S_p(\Sigma,i)$
  denote the closure of $S(\Sigma,i)$ in $V^\ast_{\Z_p}$. 
Then
\begin{equation}\label{ramanujan22}
\lim_{X\to\infty} \frac{N^\ast(S(\Sigma,i);X)}X
\,\,=\,\,
\frac{1}{2n_i^\ast}\cdot
\prod_p\Bigl(\frac{p-1}{p}\cdot \sum_{x\in S_p(\Sigma,i)/\GL_2({\Z_p})}
\frac{1}{\disc_p(x)}\cdot\frac{1}{|\Stab_{\GL_2(\Z_p)}(x)|}\Bigr),
\end{equation}
where $\disc_p(x)$ denotes the reduced discriminant of $x\in V_{\Z_p}^\ast$ as a
power of $p$ and $\Stab_{\GL_2(\Z_p)}(x)$ denotes the stabilizer of $x$ in 
$\GL_2(\Z_p)$. \end{proposition}

\begin{proof} First, note that although $S(\Sigma,i)$ might be defined by infinitely
many congruence conditions, the estimate provided
in Proposition~\ref{unifest} (and the fact that $\Sigma$ is
acceptable) shows that equation (\ref{ramanujan}) continues to hold
for the set $S(\Sigma,i)$, i.e.,
	\begin{equation*}
\lim_{X \rightarrow \infty} \frac{N^\ast(S(\Sigma,i),X)}{X} = \frac{\pi^2}{4 \cdot n_i^\ast}\prod_p \mu_p^\ast(S(\Sigma,i)).
	\end{equation*}
The argument is identical to that in \S\ref{thm1pf} or \S\ref{noncftorders}.  

If $\mu_p(S)$ denotes the $p$-adic density of $S$ in $V_\bZ$, where $\mu_p$ is normalized so that $\mu_p(V_\bZ) = 1$, then $\mu_p^\ast(S) = \mu_p(S)$ for $p \neq 3$ and $\mu_3^\ast(S) = 9 \mu_3(S)$. (This is just a reformulation of the fact that $[V_\bZ:V_\bZ^\ast] = 9$.) Thus, 
	\begin{equation*}
\lim_{X \rightarrow \infty} \frac{N^\ast(S(\Sigma,i),X)}{X} = \frac{9\cdot\pi^2}{4 \cdot n_i^\ast}\prod_p \mu_p(S(\Sigma,i)).
	\end{equation*}
By \cite[Lemma~32]{BST}, we have that
	$$\mu_p(S(\Sigma,i)) = \frac{\#\GL_2(\bF_p)}{p^4}\cdot \sum_{x \in S_p/\GL_2(\bZ_p)} \frac{1}{\Disc_p(x)}\cdot \frac{1}{|\Stab_{\GL_2(\bZ_p)}(x)|},$$
where $\Disc_p(x)$ denotes the discriminant of $x \in V_{\bZ_p}^\ast$ as a power of $p$.
Note that since $\Disc_p(x) = \disc_p(x)$ for all $p \neq 3$ and $\Disc_3(x) = 27\cdot\disc_3(x)$, we have that 
    \begin{eqnarray*}
    \lim_{X \rightarrow \infty} \frac{N^\ast(S(\Sigma,i),X)}{X} 
    &=& \frac{9\cdot\pi^2}{4 \cdot n_i^\ast}\cdot \frac{1}{27}\cdot \prod_p \frac{\#\GL_2(\bF_p)}{p^4}\cdot \sum_{x \in S_p(\Sigma,i)/\GL_2(\bZ_p)} \frac{1}{\disc_p(x)}\cdot \frac{1}{|\Stab_{\GL_2(\bZ_p)}(x)|} \\
    &=& \frac{1}{2n_i^\ast} \cdot \prod_p\left(\frac{p-1}{p}\right) \sum_{x \in S_p(\Sigma,i)/\GL_2(\bZ_p)} \frac{1}{\disc_p(x)}\cdot \frac{1}{|\Stab_{\GL_2(\bZ_p)}(x)|}.
    \end{eqnarray*}
\end{proof}

Now, if we set
\begin{equation}\label{padicmassdef}
	M_p(S(\Sigma,i)) := \sum_{x \in S_p(\Sigma,i)/\GL_2(\bZ_p)} \frac{1}{\disc_p(x)}\cdot\frac{1}{|\Stab_{\GL_2(\bZ_p)}(x)|},
\end{equation}
then the description of the stabilizer in Corollary~\ref{gl2bijection} (in its form over $\Z_p$; see Remark~\ref{rmkzp}) 
allows us to express $M_p(S(\Sigma,i))$ in another way.  
Namely, if $R\in\Sigma_p$ is a nondegenerate
quadratic ring over~$\Z_p$, then in a corresponding triple
$(R,I,\delta)$ we can always choose $I=R$, since $I$ is a principal
ideal (recall that invertible means locally principal).  
Let $\tau(R)$ denote the number of elements $\delta$, modulo cubes, yielding a
valid triple $(R,R,\delta)$ over $\Z_p$. Then 
$\tau(R)=|U^+(R)/U^+(R)^{\times3}|$, where $U^+(R)$ denotes the group of units
of $R$ having norm~1.  Since $(R,R,\delta)$ is
$\GL_2(\bZ_p)$-equivalent to the triple $(R,R,\bar\delta)$, and $\bar\delta =
\kappa^3\delta$ for some $\kappa\in R\otimes\Q$ if and only if
$\delta$ is itself a cube (since $\bar\delta=N(I)^3/\delta$), we see
that
\begin{equation}\label{massdef3}
 M_p(S(\Sigma,i)) = \sum \frac{|U^+(R)/U^+(R)^{\times3}|}
{\Disc_p(R)\cdot|\Aut(R;R,\delta)|
\cdot|U_3^+(R)|},
\end{equation}
where the sum is over all isomorphism classes of quadratic rings $R$ over
$\Z_p$ lying in $\Sigma_p$, and where $U_3^+(R)$ denotes the subgroup of 3-torsion elements of $U^+(R)$.
We have the following lemma:

\begin{lemma}\label{weird}
Let $R$ be a nondegenerate quadratic ring over $\Z_p$.  Then
\[\frac{|U^+(R)/U^+(R)^{\times 3}|}{|U_3^+(R)|}
\]
is $1$ if $p\neq 3$, and is $3$ if $p=3$.
\end{lemma}

\begin{proof}
  The unit group of $R$, as a multiplicative group, is a finitely
  generated, rank 2 $\Z_p$-module. Hence the submodule $U^+(R)$,
  consisting of those units having norm 1, is a finitely generated
  rank~1 $\Z_p$-module.  It follows that there is an isomorphism $U^+(R)\cong F \times
  \Z_p$ as $\Z_p$-modules, where $F$ is a finite abelian $p$-group.

  Let $F_3$ denote the 3-torsion subgroup of $F$.  Since $F_3$ is the
  kernel of the multiplication-by-3 map on $F$, it is clear that
  $|F/(3\cdot F)|/|F_3|=1$.  Therefore, it suffices to check the lemma on the
  ``free'' part of $U^+(R)$, namely, the $\Z_p$-module $\Z_p$,
  where the result is clear.  (The case $p=3$ differs because
  $3\cdot\Z_p$ equals $\Z_p$ for $p\neq 3$, while 
$3\cdot\Z_3$ has index 3 in $\Z_3$.)
\end{proof}

Combining (\ref{ramanujan22}), (\ref{padicmassdef}), (\ref{massdef3}), and
Lemma~\ref{weird}, we obtain 
\begin{equation}\label{ramanujan2}
\lim_{X\to\infty} \frac{N^\ast(S(\Sigma,i), X)}X
=\frac{3}{2n_i^\ast}\cdot
\prod_p\Bigl(\frac{p-1}{p}\cdot \sum_{R\in\Sigma_p}
\frac{1}{2\cdot\Disc_p(R)}\Bigr).
\end{equation}
By Corollary~\ref{hr2} and Theorem~\ref{reducible}, we have 
\begin{equation}\label{irredcountgen}
N^\ast(S(\Sigma,i), X) = 
	 \begin{cases} \displaystyle\sum_{{\cO \in \Sigma,}\atop{0<\Disc(\O)<X}} \left(3\cdot|\Cl_3(\cO)| - |\cI_3(\cO)|\right) & \mbox{if } i = 0, \\[.35in]
	 \displaystyle\sum_{{\cO \in \Sigma,}\atop{0 < -\Disc(\cO) < X}} \left(|\Cl_3(\cO)| - |\cI_3(\cO)|\right) & \mbox{if } i = 1.
	 \end{cases}
\end{equation}
Hence we conclude using Lemma~\ref{dht} that
\begin{eqnarray*} 
\frac{\displaystyle\sum_{{\cO \in \Sigma,}\atop{0<\Disc(\O)<X}} \left(|\Cl_3(\cO)| - \frac{1}{3}\cdot |\cI_3(\cO)|\right)}{\displaystyle\sum_{{\cO \in \Sigma,}\atop{0<\Disc(\O)<X}} 1} &=& \frac{1}{3}\cdot\frac{\displaystyle\frac{3}{2n_0^\ast} \cdot\prod_p\left(\frac{p-1}{p}\cdot \sum_{R\in\Sigma_p}
\frac{1}{2 \cdot \Disc_p(R)}\right)}{\displaystyle\frac{1}{2}\cdot\prod_p\left(\frac{p-1}{p}\cdot \sum_{R\in\Sigma_p}
\frac{1}{2 \cdot \Disc_p(R)}\right)} = 1, \quad \mbox{and} \\\frac{\displaystyle\sum_{{\cO \in \Sigma,}\atop{0<-\Disc(\O)<X}} \Bigl(|\Cl_3(\cO)| - |\cI_3(\cO)|\Bigr)}{\displaystyle\sum_{{\cO \in \Sigma,}\atop{0<-\Disc(\O)<X}} 1} &=& \frac{\displaystyle\frac{3}{2n_1^\ast} \cdot\prod_p\left(\frac{p-1}{p}\cdot \sum_{R\in\Sigma_p}
\frac{1}{2 \cdot \Disc_p(R)}\right)}{\displaystyle\frac{1}{2}\cdot\prod_p\left(\frac{p-1}{p}\cdot \sum_{R\in\Sigma_p}
\frac{1}{2 \cdot \Disc_p(R)}\right)} = 1,
\end{eqnarray*}
yielding Theorem~\ref{diff}.  In conjunction with 
Theorem~\ref{gensigmaord}, we also then obtain Theorem~\ref{gensigmaid}.

\subsection*{Acknowledgments}

We are very grateful to Henri Cohen, Benedict Gross, Hendrik Lenstra,
Jay Pottharst, Arul Shankar, Peter Stevenhagen, and Jacob Tsimerman
for all their help and for many valuable
discussions. The first author was supported by the Packard and Simons
Foundations and NSF Grant DMS-1001828. The second author was supported
by a National Defense Science \& Engineering Fellowship and an NSF
Graduate Research Fellowship.


\begin{thebibliography}{10}



\bibitem{Bhargava2}  
M.\ Bhargava, Higher composition laws I: A new view on Gauss
composition, and quadratic generalizations, {\it Ann.\ of Math.} {\bf 159}
(2004),  no.\ 1, 217--250.

\bibitem{Bhamass1}
M.\ Bhargava, Mass formulae for extensions of local fields, and
conjectures on the density of number field discriminants,
{\it Int.\ Math.\ Res.\ Not.}, IMRN {\bf 2007},  no.\ 17, 20 pp.



\bibitem{BST}
M.\ Bhargava, A.\ Shankar and J. Tsimerman,
On the Davenport--Heilbronn theorems and second order terms,
{\it Invent.\ Math.} {\bf 193}, 439-499.

\bibitem{CL}
H.\ Cohen and H.\ W.\ Lenstra, Heuristics on class groups of number
fields,  {\it Number theory} (Noordwijkerhout, 1983), pp.\ 
33--62, {Lecture Notes in Math.} {\bf 1068}, Springer, Berlin, 1984. 

\bibitem{Cox}
D.\ Cox, {\it Primes of the form $x^2 + ny^2$},  
John Wiley \& Sons, Inc., New York, 1989. 

\bibitem{DW}
B.\ Datskovsky and D.\ J.\ Wright, Density of discriminants of cubic
extensions,  {\it J.\ Reine Angew.\ Math.}  {\bf 386}  (1988), 116--138.

 
\bibitem{Davenport1}
H.\ Davenport, On the class-number of binary cubic forms I,
{\it J.\ London Math.\ Soc.} {\bf 26} (1951), 183--192.

\bibitem{Davenport2}
H.\ Davenport, On the class-number of binary cubic forms II,
{\it J.\ London Math.\ Soc.} {\bf 26} (1951), 192--198.
 
\bibitem{DH} 
H.\ Davenport and H.\ Heilbronn, On the density of discriminants of
cubic fields II, {\it Proc.\ Roy.\ Soc.\ London Ser.\ A} {\bf 322} (1971), 
no.\ 1551, 405--420.

\bibitem{DF}
B.\ N.\  Delone and D.\ K.\ Faddeev, {\it The theory of irrationalities
of the third degree}, AMS Translations of Mathematical Monographs
{\bf 10}, 1964.

\bibitem{Eisenstein}
G.\ Eisenstein, Aufgaben und Lehrs\"atze, {\it J.\ reine angew.\ Math.}, {\bf 27} (1844), 89--106.

\bibitem{GGS}
W.-T.\ Gan, B.\ H.\ Gross, and G.\ Savin, 
Fourier coefficients of modular forms on $G_2$, {\it Duke Math. J.}
{\bf 115} (2002), 105--169.

\bibitem{Nakagawa} 
J.\ Nakagawa, On the relations among the class numbers of binary cubic forms, {\it Invent.\ Math.} {\bf 134} (1998), 101--138.

\bibitem{Shintani}
T.\ Shintani, On Dirichlet series whose coefficients are class-numbers
of integral binary cubic forms, {\it J.\ Math.\ Soc.\ Japan} {\bf 24}
(1972), 132--188.


\end{thebibliography}
\end{document}